%% file: multicircle_september_25.tex
\documentclass{article}
\usepackage{latexsym}
\usepackage{amsmath,amsthm}
\usepackage{amssymb}
\usepackage{bbm}
\usepackage{color}
\usepackage{xcolor}
\usepackage{hyperref}
\usepackage{graphicx}
\usepackage{epsfig}
\usepackage[tmargin=1.25in,bmargin=1.25in,lmargin=1.1in,rmargin=1.1in]{geometry}
\usepackage{booktabs}
\usepackage{mdframed}
\usepackage{amsmath, amssymb}
\usepackage{nomencl}

\newcommand{\sn}{^{(n)}}
\newcommand{\tr}{\mathrm{tr}}

\def \sur#1#2{\mathrel{\mathop{\kern 0pt#1}\limits^{#2}}}

\newcommand{\UNIF}{\mathrm{UNIF}}

\newcommand{\GW}{\mathrm{GW}}

\newcommand{\coeff}{\mathrm{co}}
\newcommand{\meas}{\mathrm{sp}}

\newcommand{\g}{\mathrm{g}}
\newcommand{\dd}{\mathrm{d}}
\newcommand{\ddd}{\mathrm{d}}
\newcommand{\w}{\mathrm{w}}
\newcommand{\V}{\mathcal V}

\newcommand{\NomBar}{\ensuremath{\mid}}

\newtheorem{theorem}{Theorem}[section]

\newtheorem{corollary}[theorem]{Corollary}
\newtheorem{lemma}[theorem]{Lemma}

\newtheorem{proposition}[theorem]{Proposition}

\newtheorem{remark}[theorem]{Remark}

\makeatletter\@addtoreset{equation}{section}\makeatother

\newcommand{\ea}{\end{array}}
\newcommand{\beqohne}{\begin{eqnarray*}}
\newcommand{\eeqohne}{\end{eqnarray*}}
\newcommand{\beohne}{\begin{equation*}}
\newcommand{\eeohne}{\end{equation*}}
\newcommand{\R}{\mathbb{R}}
\newcommand{\N}{\mathbb{N}}
\newcommand{\T}{\mathbb{T}}

\newcommand{\C}{\mathbb{C}}
\newcommand{\D}{\mathbb{D}}

\newcommand{\PP}{\mathbb{P}}

\def\proof{\noindent{\bf Proof:}\hskip10pt}
\def\QED{\hfill\vrule height 1.5ex width 1.4ex depth -.1ex \vskip20pt}
\def \sur#1#2{\mathrel{\mathop{\kern 0pt#1}\limits^{#2}}}

\def \w{{\tt w}}

\def \ben{\begin{eqnarray}}
\def \een{\end{eqnarray}}

\newcommand{\Sr}{\mathcal{S}}

\newcommand{\fab}{\color{black}}

\DeclareMathOperator{\GE}{HP}

\definecolor{Red}{rgb}{1,0,0}
\definecolor{Blue}{rgb}{0,0,1}

\newcommand{\supp}{\operatorname{supp}}

\makenomenclature

\begin{document}

\title{Sum rules via large deviations: polynomial potentials and multi-cut regime on the unit circle}
\author{Fabrice Gamboa\thanks{Universit\'e de Toulouse, Institut de Math\'ematiques de Toulouse,  31062-Toulouse Cedex 9, France and ANITI,{\tt fabrice.gamboa@@math.univ-toulouse.fr}} \and Jan Nagel\thanks{Technische Universit\"at Dortmund, Fakult\"at f\"ur Mathematik, 44227 Dortmund, Germany,{\tt jan.nagel@tu-dortmund.de}} \and Alain Rouault\thanks{Laboratoire de Math\'ematiques de Versailles, UVSQ, CNRS, Universit\'e Paris-Saclay, 78035-Versailles Cedex France,{\tt alain.rouault@uvsq.fr}}}
%

\maketitle
\begin{abstract}
Sum rules are elegant formulas that relate entropy functionals to coefficients associated with orthogonal polynomials \cite{simonszego}.
In a series of paper (see for example \cite{gamboanag2016}, \cite{gamboanag2017}, \cite{breuersizei2017}, \cite{breuersizei2018}), interesting connections have been established between the large theory of spectral measures built on random matrices and sum rules.
In this work, we extend this approach by studying sum rules within random matrix models with polynomial potentials on the unit circle, with a particular focus on cases where the equilibrium measure lacks full support.
\end{abstract}


\tableofcontents

\section{Introduction}
\label{sec:introduction}

This paper deals with sum rules arising from spectral theory and orthogonal polynomials on the unit circle (OPUC).  
The first historical example of such a sum rule is the classical Szeg{\H o}-Verblunsky identity, 
\begin{equation}
\label{SVsumrule}
\frac{1}{2\pi}\int_{0}^{2\pi}\log g_{\mu}(z)\ddd z = \sum_{k= 0}^\infty \log(1-|\alpha_k|^2)\,,
\end{equation}
where $\mu$ is a {\fab probability }measure on the unit circle $\mathbb T = \{z = e^{i\theta}: \theta \in [0,2\pi)\}$ with Lebesgue decomposition 
\[\ddd\mu(z)= g_{\mu}(z) \tfrac{\ddd z}{2\pi}+\ddd\mu_s(z)\]
and $(\alpha_k)_{k\geq 0}$\nomenclature{$(\alpha_k)$}{Verblunsky coefficients (V-coefficients)} are the coefficients appearing in the recursion of {\fab orthogonal} polynomials  with respect to $\mu$, also called {\fab the}  Verblunsky coefficients. In the sequel we will say V-coeffcients, for short. Both sides of \eqref{SVsumrule} vanish if, and only if, $\mu$ is the uniform measure on the circle ({\fab this is the reference measure here}). 
Furthermore, both sides are finite or infinite simultaneously. 
 Changing the signs in both sides of this equation leads to the formulation
\begin{equation}
\label{SVsum}
\mathcal K(\operatorname{UNIF} | {\fab\mu}) = -\sum_{k \geq 0} \log(1 - |\alpha_k|^2)
\end{equation}
where $\operatorname{UNIF}$ is the normalized Lebesgue measure on  $\mathbb T$ and for probability measures $\nu$ and $\mu$, $\mathcal K (\nu | \mu)$ denotes the Kullback-Leibler divergence or relative entropy of $\nu$ with respect to $\mu$ \begin{equation}
\label{kukul}
\mathcal{K}(\nu|\mu) = \int \log \left( \frac{\ddd \nu}{\ddd\mu} \right) \ddd\nu 
\end{equation}
if $\nu$ is absolutely continuous with respect to $\mu$ and $\log\frac{\ddd \nu}{\ddd\mu}\in L^1(\nu)$, and   $\mathcal{K}(\nu|\mu)=\infty$
\nomenclature{$\mathcal{K}(\cdot \mid\cdot)$}{Kullback-Leibler divergence}
\nomenclature{$\operatorname{UNIF}$}{Normalized Lebesgue measure} 
otherwise. 

In \cite{simonopuc1} B. Simon gave (at least) four 
analytical proofs of (\ref{SVsumrule}). An important consequence of a sum rule is the existence of equivalent conditions for the finiteness of either side; in the words of Simon, these are the \emph{gems} of spectral theory. The gem implied by \eqref{SVsum} is
\begin{align}
\label{SVgem}
\mathcal K(\operatorname{UNIF} | \mu) < \infty \quad \Leftrightarrow\quad \sum_{k=0}^\infty |\alpha_k|^2<\infty. 
\end{align}
Since the classical work of Szeg\H{o} and Verblunsky, there has been a great interest in extending such a sum rule and gem, and to replace the uniform law $\UNIF$ by a more general reference measure, for example in 
\cite{simonopuc1,simon2005higher,golinskiizlat2007,lukic2016higher,yan2018,du2023}.
A common feature of these generalizations is that the reference measure has a density which is polynomial in $\cos (\theta)$ and is supported by the full circle $\mathbb T$.

{\fab In \cite{gamboanag2016}, we developed a probabilistic method for establishing sum rules for probability measures supported on the real line. This approach builds on the seminal work \cite{gamboarou2010}, and was later extended to probability measures on the unit circle $\mathbb{T}$ in \cite{gamboanag2017}. We also refer the reader to \cite{breuersizei2017} for an introductory exposition of the method. 
The core of the method relies on a large deviation principle for a sequence of random measures associated with random matrices. The measure $\mu$ is interpreted as the realization of a (random) spectral measure corresponding to a pair $(M, e)$, where $M$ is a random unitary operator and $e$ is a fixed cyclic vector in a Hilbert space $\mathcal{H}$.}
If  $\dim\mathcal H = n \geq 1$,
 then $\mu$ is a discrete probability measure on the unit circle $\T$ which can be written as 
\begin{equation}
\label{meumeu}
\mu_n = \sum_{k=1}^n \w_k \delta_{e^{i\theta_k}}\,,
\end{equation}
where $e^{i\theta_1},\dots ,e^{i\theta_n}$ are the eigenvalues of $M$ and $\w_k=|\langle e,u_k\rangle |^2$ with $u_1,\dots ,u_n$ a corresponding set of orthonormal eigenvectors. 
A classical assumption is the invariance  of the law of $M$ under any unitary conjugations. Under this assumption,  the joint density of $(e^{i\theta_1}, \dots, e^{i\theta_n})$ is  proportional to
\begin{align} \label{invariantevdensity}
|\Delta(e^{i\theta_1}, \dots, e^{i\theta_n})|^2 \exp \left( -n \sum_{k=1}^n \mathcal V(e^{i\theta_k})\right)
\end{align}
where $\Delta$  is the Vandermonde determinant and $\mathcal V$ {\fab is} some potential. \nomenclature{$\Delta(e^{i\theta_1}, \dots, e^{i\theta_n})$}{Vandermonde determinant}
\nomenclature{$\mathcal V$}{Potential}
\nomenclature{$\mu_{\mathcal V}$}{Equilibrium measure associated to $\mathcal V$}
 {\fab In this frame},  the distribution of the weights $(\w_1, \dots, \w_n)$ is uniform on the simplex. With convenient assumptions on $\mathcal V$, the sequence $(\mu_n)_n$ converges to a deterministic measure $\mu_\V$ on $\T$, called the equilibrium measure.  Actually  $(\mu_n)_n$  satisfies, as $n$ grows, a large deviation principle (LDP). Encoding the spectral measure by its support points and weights, it is possible to show an LDP with rate function $\mathcal{I}_{\meas}$. If it also possible to study the large deviations from the distribution of random V-coefficients of $\mu_n$, we obtain a rate function $\mathcal I_{\coeff}$. {\fab Since the rate function is unique, combining both approaches yields the identity $\mathcal{I}_{\meas} = \mathcal{I}_{\coeff}$, which constitutes a potential sum rule.}
%

This scheme is the most simple in the Circular Unitary Ensemble, for which $\V = 0$, the equilibrium measure is the uniform law $\mu_\V = \UNIF$  and the $\alpha_k$ are independent. In this case $\mathcal I_{\meas}$ and $\mathcal I_{\coeff}$ is the left and right hand side, respectively, of \eqref{SVsum}, which leads to a probabilistic proof of the Szeg\H{o}-Verblunsky sum rule, see \cite{gamboanag2017} or \cite{breuersizei2017} for a spectral theoretic exposition. {\fab For other potentials, although the large deviation principle on the measure side is relatively straightforward to establish, the $V$-coefficients become dependent, making their large deviation analysis significantly more challenging. The Hua--Pickrell model, corresponding to the potential $\V(z) = -\mathrm{d} \log |1 - z|$, is also studied in \cite{gamboanag2017}. In that work, we use deformed $V$-coefficients, which are independent.
}

 
Using this probabilistic approach, Breuer, Simon and Zeitouni \cite{breuersizei2018} considered  potentials $\V$ which are symmetric Laurent polynomials in $z= e^{i\theta}$ of degree $d$. They 
prove the LDP for the coefficient side when the measure is fully supported by $\T$ (ungapped case).  The V-coefficients are no longer independent, instead, their dependence structure exhibits finite-range (finite-memory) correlations.
As particular examples, they considered   
ensembles whose equilibrium measures  have a density $\rho^{\alpha,\beta}$ proportional to 
\[(1- \cos \theta)^\alpha (1 + \cos \theta)^\beta\] for $ (\alpha, \beta) = (1,0), (1,1)$ and $(2,0)$, respectively. 
\nomenclature{$\rho^{\alpha,\beta}$}{Probability density proportional to $(1- \cos \theta)^\alpha (1 + \cos \theta)^\beta$ }

It is common in the study of particle systems and in random matrix theory to multiply the potential $\V$ in \eqref{invariantevdensity} by a parameter $\g$, which is interpreted as the inverse temperature and measures the strength of the coupling. For $|\g|$ small, the equilibrium measure is supported by $\mathbb T$, whereas for large $|\g|$ the equilibrium measure typically is supported by a proper subset of the circle. Starting from the potential
\[\V (z) = \frac{1 }{2}(z + z^{-1})\] whose equilibrium density is supported by $\T$ and is given by 
\[\rho^{1,0}(e^{i\theta}) = \frac{1}{2\pi} (1- \cos \theta)\,,\]
many authors considered the potential $\g \V^{1,0}$ and the corresponding density $\rho^{1,0}_\g$.
For $|\g| \leq 1$ (strong coupling) the density is supported by $\T$:
\begin{align}
\label{defrhog}
\rho_\g^{1,0}(e^{i\theta}) =\frac{1}{2\pi} (1- \g \cos \theta)\,.
\end{align}
This model is commonly known as Gross-Witten model. 

For  $|\g| > 1$ (weak coupling) the support of  the equilibrium measure is a proper arc of the unit circle (gapped phase). 
This model  is important in the analysis of problems involving random permutations \cite{baikdeiJohan1999}. 
For details and applications  we refer to \cite[p. 203]{HiaiP}, \cite{grosswi1980}, \cite{wadia2012}.


{\fab The purpose of this paper is to extend the probabilistic method in order to establish sum rules and identify gems for reference measures whose support is a proper subset of $\mathbb{T}$.  
We focus on general polynomial potentials whose equilibrium measures are supported on one or several proper arcs of the unit circle---commonly referred to as the one-cut and multi-cut cases.}
%
%
{\fab Regarding large deviations, a major distinction from the earlier situation comes from the presence of outliers—eigenvalues outside the support of the equilibrium measure $\mu_{\V}$. These outliers introduce an additional contribution on the measure side, governed by the effective potential.  
The large deviation principle can be derived by transferring a recent result from the real line to the circle, as done in \cite{gamboanag2021}.}
%
{\fab On the coefficient side, it is necessary to revisit the proof by Breuer, Simon, and Zeitouni presented in \cite{breuersizei2018}. Careful bounds on the coefficient density then provide estimates for the corresponding rate function.}

The main results of this article are a general gem for the multi-cut case (Theorem \ref{Unitabstractgem}) and a sum rule in the one-cut case (Theorem \ref{Unitabstractnewsumrule}). The sum rule comes with the caveat that it requires a specific convergence of the V-coefficients. This is not an artifact of our proof and we argue in Remark \ref{rem:counterexample}, that a general sum rule with a polynomial right hand side cannot hold for any measure $\mu$ on $\T$. For symmetric one-cut measures we 
show in Theorem \ref{Unitnewsumrulesymmetric} that the sum rule always holds for a specific representative in the Aleksandrov-class of $\mu$ (see \cite{simonopuc1} for the Aleksandrov-class definition).  

As an application, we consider the three models  $(\alpha,\beta)=(1,0), (1,1)$ and $(\alpha,\beta)=(2,0)$ parametrized by $\g$  for which we have the same phenomenon with ungapped/gapped phases according to the value of $\g$.  
We state generalized sum rules and gems in the ungapped phases and in the gapped phase we find a two-cut regime for  (1,1) and a one-cut regime for (2,0). For the Gross-Witten model we obtain a precise gem.

The paper,  which can be seen as a companion paper of \cite{gamboarou2010}, \cite{gamboanag2017} and \cite{gamboanag2021}, is organized as follows. In the next section we give some notations, assumptions and some connections with the work of Breuer-Simon-Zeitouni \cite{breuersizei2018} and Du \cite{du2023}.  In Section \ref{sec:BSZworkshop}, we give the previous results of Breuer, Simon and Zeitouni for the ungapped case. Section \ref{sec:our_results} contains our main results: the abstract gem in the general case and a sum rule in the one-cut case.  In Section \ref{sec:one-par-ungapped}, we examine the ungapped phases of our three models parametrized by $\gamma$. Section \ref{sec:one-par-gapped} extends this analysis to the gapped phase.
In Section \ref{sec:LDP}, we state the probabilistic large deviation results behind the gem and sum rule and explain how they imply the sum rules.  
Section \ref{sec:proofs} is devoted to the proofs of the large deviation results.
The appendix concludes the paper by providing explicit computations for specific potentials, further details on the derivation of the sum rules in Section \ref{sec:one-par-ungapped}, and additional auxiliary entropy calculations.

\section{Notations and definitions}
\label{sec:sumrules}

Let us denote by $\mathcal M_1(\T)$ the set of all probability measures on $\T$. We consider {\fab orthogonal polynomials  for some} measure $\mu\in \mathcal M_1(\T)$ and their recursion. If
$\mu$ is supported by an infinite set,   
we construct the monic orthogonal polynomials $\Phi_k, k\geq 0$ by the Gram-Schmidt procedure. The recursion of these polynomials involves the reversed polynomials 
\begin{equation*}
\label{reversed}
    \Phi_k^\star (z) = z^k \overline{\Phi_k(1/\bar z)} .
\end{equation*}
Notice that $\Phi_k^\star$ is the unique polynomial of degree at most $k$,
orthogonal to $z, z^2, \cdots, z^k$ and such that $\Phi_k^\star (0)=1$ {\fab (see Lemma 1.5.1. in \cite{simonopuc1})}.
A sequence of normalized polynomials is obtained by setting $\varphi_0=1$ and for $k\geq 1$
\begin{equation}
   \label{orthonc}
\varphi_k = \frac{\Phi_k}{\Vert\Phi_k\Vert} .
\end{equation}
The associated reversed polynomials are
\begin{equation}
 \label{orthoncs}
   \varphi_n^\star = \frac{\Phi_n^\star}{\Vert\Phi_n\Vert}\,.
\end{equation}
The \emph{Szeg\H{o} recursion} of the monic polynomials is then the relation
\begin{equation}
\label{recOPUC}
\Phi_{k+1}(z)=z\Phi_{k}(z)-\overline{\alpha}_k\Phi_{k}^*(z),
\end{equation}
with a recursion coefficient 
\[\alpha_k=-\overline{\Phi_{k+1}(0)}\]
 the so-called Verblunsky coefficient or $V$-coefficient, for short. 

From \eqref{recOPUC}, we obtain the recursion
\begin{align} \label{OPUCnorm}
||\Phi_{k}||^2_{L^2(\mu)} = (1-|\alpha_{k-1}|^2) ||\Phi_{k-1}||^2_{L^2(\mu)} = \prod_{i=0}^{k-1} (1- |\alpha_i|^2),  
\end{align}
showing that $\alpha_k \in \mathbb D= \{z\in \mathbb{C}: |z|<1\}$ for all $k\geq 0$. For later reference, we set 
\begin{align} \label{defrho}
\rho_k = \sqrt{1-|\alpha_k|^2} .
\end{align} 
If $\mu\in \mathcal{M}_1(\mathbb{T})$ is supported by $n$ points,
\[\mu= \sum_{k=1}^n \w_k\delta_{\lambda_k}\]
 we can only define monic polynomials up to degree $n$, which by \eqref{recOPUC} defines V-coefficients $\alpha_0,\dots ,\alpha_{n-1}$. In this case, formula \eqref{OPUCnorm} is still valid up to $k=n$. 
From \eqref{OPUCnorm} we see that $\alpha_k \in \mathbb{D}$ for $1\leq k \leq n-2$, but $\alpha_{n-1}\in \mathbb{T}$.  {\fab Before ending this section notice that we adopt the convention $\alpha_{-1} = -1$}

\subsection{Spectral measure and matrix representations}
According to  the spectral theorem,  
if $\mathcal H$ is an Hilbert space and if $U$ is a bounded unitary operator on $\mathcal H$, there exists a unique operator-valued measure $\mathcal E$ on $\T$
such that
\begin{align}
\label{spthcomplex}
U = \int_\T \lambda\, \ddd\mathcal E(\lambda).
\end{align}
A vector $e\in \mathcal H$ is called cyclic if  span $(e, Ue, \cdots) = \mathcal H$. In this case the spectral measure of the pair $(U,e)$ is the measure $\mu= \langle e, \mathcal E(\cdot) e\rangle$. It is  uniquely determined by its moments
\begin{align}
\label{defspectralmeasurecomplex}
\int_\mathbb T z^k \ddd \mu(z) = \langle e, U^k e\rangle ,\qquad k\in \mathbb{Z}.
\end{align}
If dim $\mathcal H =n$ and if $e$ is cyclic for $U$, we let $\lambda_1, \ldots, \lambda_n$ denote the eigenvalues of $U$ (generically distinct), belonging to $\T$,  and $\psi_1, \ldots, \psi_n$ a system of associated orthonormal eigenvectors. Then, the spectral measure of the pair $(U,e)$ is 
\begin{align}\label{spectralmeasure}
\mu_n =  \sum_{k=1}^n \w_k\delta_{\lambda_k}\ , \ \w_k= |\langle \psi_k, e\rangle|^2 \ (k=1, \dots, n)\,.
\end{align}
It is a weighted version of the empirical eigenvalue measure
\begin{align*}
\hat\mu_n =  \frac{1}{n}\sum_{k=1}^n \delta_{\lambda_k} .
\end{align*}

There are two important matrix representations for $\mu\in \mathcal{M}_1(\mathbb{T})$, the GGT matrix and the CMV matrix. 

The GGT matrix, named after Geronimus, Gragg, and Teplyaev \cite[Sec. 4.1]{simonopuc1}, is the matrix of the mapping $f \mapsto \mathrm{Shift} f$, where $(\mathrm{Shift} f)(z) = z f(z)$ on $L^2(\mu)$, equipped 
 with  the basis $(\varphi_k)_{k\geq 0}$ defined in \eqref{orthonc}. It is  
\begin{align}
\label{GGTm}
\mathcal G_\mu  = 
\begin{pmatrix} \bar\alpha_0 &\rho_0\bar\alpha_1 &\rho_0 \rho_1 \bar\alpha_2 & \rho_0\rho_1 \rho_2 \bar\alpha_3 & \cdots\\
\rho_0&-\alpha_0\bar\alpha_1 & -\alpha_0\rho_1 \bar\alpha_2 &- \alpha_0\rho_1 \rho_2 \bar\alpha_3 & \cdots\\
0 & \rho_1 &  - \alpha_1\bar\alpha_2 &- \alpha_1\rho_2 \bar\alpha_3 & \cdots\\
0&0 &\rho_2 & - \alpha_2\bar\alpha_3  & \cdots\\
\vdots&\vdots&\vdots& \vdots&\ddots
\end{pmatrix}\,.
\end{align}
When $\mu$ has an infinite support, $\mathcal G_\mu$ is an infinite matrix and when 
 $\mu$ is supported by $n$ points $\mathcal G_\mu$ is an $n\times n$ matrix. In this latter case, the last line is
\[0\mid \dots \mid 0 \mid \rho_{n-2} \mid -\alpha_{n-2}\bar\alpha_{n-1}\]
and  last column 
\[\left(\rho_0\rho_1 \dots\rho_{n-2}\bar\alpha_{n-1} \mid -\alpha_0\rho_1\dots\rho_{n-2}\bar\alpha_{n-1}\mid \cdots\mid -\alpha_{n-2}\bar\alpha_{n-1}\right)^T\]
If a (finite or infinite) sequence $\alpha$ is given, we denote by $\mathcal G(\alpha)$ the associated GGT matrix.  We denote by $\mathcal G_L(\alpha)$ the upper left $L\times L$ section of $\fab\mathcal G_{\mu}$, $\fab(L\in\N_*)$.
\nomenclature{$\mathcal G_L(\alpha)$}{Upper left $L\times L$ section of $\mathcal G_{\mu}$}
The five-diagonal or CMV representation is due to Cantero, Moral, and Vasquez \cite{canteromoral2003}, when using the new basis $(\chi_k)_{k \geq 0}$, with 
\begin{align}
\label{defCMVbasis}
\chi_{2k}(z) = z^{-k}\varphi_{2k}^{\star}(z) ,\qquad 
\chi_{2k-1}(z) = z^{-k+1}\varphi_{2k-1}(z).  
\end{align}
It is obtained by orthonormalizing $1, z , z^{-1}, z^2, z^{-2}, \dots$. {\fab In this new Hilbertian basis, the $\mathrm{Shift}$ mapping}  on $L^2(\mu)$ is represented by the matrix 
\begin{align}
\label{favardfiniC}
\mathcal C_\mu = \begin{pmatrix} \bar \alpha_0&\bar \alpha_1\rho_0 &\rho_1\rho_0&0&0&\cdots\\
\rho_0& -\bar\alpha_1\alpha_0&-\rho_1\alpha_0 &0 &0 &\cdots\\
0&\bar\alpha_2\rho_1&-\bar\alpha_2\alpha_1&\bar\alpha_3\rho_2&\rho_3\rho_2&\cdots\\
0&\rho_2\rho_1& -\rho_2\alpha_1& -\bar\alpha_3\alpha_2&-\rho_3\alpha_2&\cdots\\
0&0&0&\bar\alpha_4\rho_3&-\bar\alpha_4\alpha_3& \cdots\\
\vdots&\vdots&\vdots&\vdots&\vdots&\ddots
\end{pmatrix} .
\end{align}
When $\mu$ has infinite support, $\mathcal{C}_\mu$ is an infinite matrix and when $\mu$ is supported by  $n$ points, $\mathcal{C}_\mu$ is $n\times n$ with last line 
\begin{eqnarray*}
\begin{cases}
0\mid \cdots\mid 0 \mid 0 \mid \bar\alpha_{2r}\rho_{2r-1} \mid  - \bar\alpha_{2r}\alpha_{2r-1} \ \ &\hbox{if} \ n=2r+1,\\
0\mid \cdots\mid  0\mid  \rho_{2r}\rho_{2r-1}\mid  - \rho_{2r}\alpha_{2r-1}\mid  -\bar\alpha_{2r+1}\alpha_{2r}\ \ &\hbox{if} \ n=2r+2 , \ r\geq 0. 
\end{cases}
\end{eqnarray*}
The measure $\mu$ is then the spectral measure of $(\mathcal{G}_\mu,e_1)$ and of $(\mathcal{C}_\mu,e_1)$, meaning 
\begin{align} \label{spectralmeasureGCmatrices}
\int_\mathbb T z^k \ddd \mu(z) = \langle e_1,\mathcal{G}^k_\mu e_1\rangle = \langle e_1,\mathcal{C}^k_\mu e_1\rangle . 
\end{align}

\medskip

As a consequence, we have a one-to-one relation between a measure $\mu$ and its sequence of recursion coefficients. To be more precise, let $\mathcal R_\infty = \D^{\mathbb{N}_0}$ be the set of all V-coefficient sequences of measures with infinite support. For $n\in \mathbb{N}$, let   
\begin{align*}
\mathcal R_n = \D^{n-1}\times \T 
\end{align*}
be the set of V-coefficients of measures with $n$ support points, and set
\begin{align*}
\mathcal R = \mathcal R_\infty \cup \left(\bigcup_{n=1}^\infty \mathcal R_n\right) .
\end{align*}
For $\alpha \in \mathcal R$, let $N(\alpha) = \inf\{n\geq 1: \alpha_n \in \T\}$, such that $\alpha \in \mathcal R_{N(\alpha)}$, with $N(\alpha)\in \mathbb{N}\cup \{\infty\}$. 
 
On $\mathcal R$, we introduce the following metrizable topology, following \cite{breuersizei2017}. A sequence $(\alpha^{(n)})_{n\geq 1}$ converges in $\mathcal R$ to $\alpha$, if $\alpha_k^{(n)}\to \alpha_k$ for all $k<N(\alpha)+1$. Note that in general $N(\alpha^{(n)})\not\to N(\alpha)$, but $N(\alpha)$ is uniquely determined by the sequence $(\alpha^{(n)})_{n\geq 1}$. When $\mathcal M_1(\T)$ is equipped with the weak topology, the mapping 
\begin{align}\label{defphimapping}
\psi: \mathcal M_1(\T) \longrightarrow \mathcal R, 
\end{align}
{\fab that maps} a measure to its V-coefficients, is then a homeomorphism. 

\medskip

We will frequently work with a selection of coefficients and introduce the following notation. For $p,q,n \in \mathbb N_0, p<q, n \geq 1$, let
\[ [p,q] = \{p, p+1, \cdots, q\} ,\quad  [n] = [0, n]\,.\]

\nomenclature{$[p,q]$}{$\{p, p+1, \cdots, q\}$}
\nomenclature{$[n]$}{$[0, n]$}
For a vector $x=(x_k)_{k\geq 0}$ we then define
\begin{align*}
x_{[p,q]} = (x_p,\dots ,x_q), \quad x_{[n]} = (x_0,\dots x_n) .
\end{align*} 
\nomenclature{$x_{[p,q]}$}{$(x_p,\dots ,x_q)$}
\nomenclature{$x_{[n]}$}{$(x_0,\dots x_n)$}
The projection onto the first $n$ coordinates is denoted by $\pi_n$, such that $\pi_n(x) = x_{[n]}$.


\subsection{Randomization}
\label{sec:randomization}

On the unitary group $\mathbb U(n)$, we denote by $\mathbb P\sn$  the normalized Haar measure, also called the Circular Unitary Ensemble (CUE). It is classical that under $\mathbb P\sn$ the array of eigenvalues of $U\in \mathbb U(n)$ has a density with respect to the Lebesgue measure $\dd\zeta_1 \dots \dd\zeta_n$ on $\mathbb T^n$  which is proportional to 
\[\left|\Delta(\zeta_1, \dots, \zeta_n)\right|^2 \, ,\]
where $\Delta$ is the Vandermonde determinant. 
Moreover, due to the invariance of the Haar measure, the matrix of eigenvectors is again Haar distributed and independent of the eigenvalues. {\fab So that the vector,} $(\w_1, \dots, \w_n)$ as defined in  (\ref{spectralmeasure})   is uniformly distributed on the $n$-simplex and independent of $(\zeta_1, \dots, \zeta_n)$.
The law of the  V-coefficients is surprisingly explicit and given by the famous Killip-Nenciu theorem.

\begin{theorem}\cite[Theorem 1]{killipnen2004}
\label{KN}
Under $\mathbb P\sn$, the distribution of the random V-coefficients $\alpha\sn :=\left(\alpha\sn_0,\dots , \alpha\sn_{n-1}\right)$ has a density with respect to the  Lebesgue measure on $\D^{n-1}\times \T$ given by
\begin{equation}
\label{ref}
\prod_{k=0}^{n-2}  \frac{n-k-1}{\pi}\left(1 - |\alpha_{\fab k}|^2\right)^{n-k-2}.
\end{equation}
\end{theorem}

%
%
%
%

More generally, it is usual to equip $\mathbb{U}(n)$ with a probability measure of the form
\begin{align}
\label{PnV}
\ddd\mathbb P\sn_\V (U) = \frac{1}{\mathcal Z_n^\V} e^{-n \tr \V(U)} \ddd\mathbb P\sn (U)\,\fab .
\end{align}
{\fab Here,} the potential $\V:\T\to (-\infty,\infty]$ satisfies a convenient integrability assumption and  $\mathcal Z_n^{\mathcal V}$ is  the normalizing constant.
The density of eigenvalues under $\mathbb P_\V\sn$ is then proportional to 
\begin{equation}
\label{VdMC}
\left|\Delta(\zeta_1, \dots, \zeta_n)\right|^2 \exp \left( - n \sum_{i=1}^n \V(\zeta_i)\right)\, .\end{equation}
By invariance, the array $(\w_1, \dots, \w_n)$  is still uniformly distributed on the simplex and independent of the eigenvalues.

Using the same transformation as in Theorem \ref{KN}, we see that under $\PP_\V\sn$, the distribution $P_\V\sn$ of $\alpha\sn$ has a density
 with respect to the Lebesgue measure on $\D^{n-1}\times \T$ which is
\begin{align}
\label{krishnaT}
\left(\tilde Z_n^{\V}\right)^{-1} \exp  \left\{-n \left(\tr \mathcal V({\fab\mathcal G_n(\alpha)})- \sum_{k=0}^{n-2}\left(1-\frac{k+2}{n}\right)\log(1 - |\alpha_k|^2)\right)\right\}\,,
\end{align}
where $\tilde Z_N^{\V}$ is the normalization constant.

We always make the following assumption of the potential $\V$:  
\begin{itemize}
\item[(A1)] $\V$ is finite and continuous on $\T\setminus \{1\}$, with existing limit $\lim_{z\to 1} \V(z)\in (-\infty,\infty]$ in 1. 
\end{itemize}
This assumption implies the existence of 
a unique minimizer $\mu_\V$ of 
\begin{equation}
\label{4.70}
\mu \mapsto \mathcal E_\V (\mu) = \int_{\mathbb T} \V (z) \ddd\mu(z) - \int_{\mathbb T^2} \log|z-\zeta| \ddd\mu(z)\ddd\mu(\zeta)\,,\qquad \mu \in \mathcal{M}_1(\T),
\end{equation}
{\fab (see \cite{saff1997logarithmic} Theorem 1.3).}
The measure $\mu_\V$ is called the equilibrium measure and it is the weak limit of the spectral measure $\mu_n$ and of the empirical eigenvalue measure $\hat\mu_n$. 
We will suppose that either the support of $\mu_\V$ is $\mathbb T$ (no cut) or: 
\begin{itemize} 
\item[(A2)] One-cut regime: the support of $\mu_\V$ is a single arc $\hat{a}$.
\item[(A2')] Multi-cut regime: the support of $\mu_\V$ is a finite union of arcs $\hat{a}_1,\dots ,\hat{a}_m$.
\end{itemize}
{\fab Note that under assumption (A2), two cases may arise: either \( \hat{a} \) is a subset of \( [-\pi, \pi) \), or it is a subset of \( [0, 2\pi) \), depending on whether \( 0 \in \hat{a} \) or not.\footnote{We will identify \( \hat{a} \) with its image on \( \mathbb{T} \).}}

Under either (A2) or (A2'), $\mu_\V$ is characterized by the Euler-Lagrange equations, 
\begin{equation}
\label{ELT}
\mathcal J_\V(z)
\begin{cases} = 2\xi_\V & \hbox{if}\ z \in \supp (\mu_\V), \\
\geq 2 \xi_\V & \hbox{otherwise},
\end{cases}
\end{equation}
where $\mathcal J_\V$ is the effective potential
\nomenclature{$\mathcal J_\V$}{Effective potential}
\begin{align}
\label{poteffT}
\mathcal J_\V (z) := \V(z) -2\int_{\mathbb T} \log |z-\zeta|\!\ \ddd\mu_\V(\zeta)\,,
\end{align}
and $\xi_\V$ is the modified Robin constant. 
{\fab We refer to \cite{saff1997logarithmic} for the definitions and quantities mentioned above, which are drawn from classical potential theory.
Analogously to the real line setting studied in \cite{gamboanag2021}, we impose the following assumption.
}
\begin{itemize}
\item[(A3)] Control (of large deviations): $\mathcal J_\V$ achieves its global minimum value on the complement of  $\operatorname{Int}(\supp (\mu_\V))$ only on the boundary of this set.
\end{itemize}
{\fab Under this assumption,} we set
\begin{align} \label{defeffectivepot}
\mathcal F_{\V}(z) = \mathcal J_\V(z) - \inf
\mathcal J_\V(\zeta)\,.
\end{align}
{\fab We also observe that if the function \( \theta \mapsto v(\theta) := \mathcal{V}(e^{i\theta}) \) is convex and the support of \( \mu_{\mathcal{V}} \) is not the entire unit circle \( \mathbb{T} \), then assumptions (A2) and (A3) are satisfied (see, for example, Section 2.3.2 in \cite{gamboanag2017}).
A particularly important example is that of a symmetric Laurent polynomial potential,}
\begin{align}
\label{even}\mathcal V(z) = \sum_{l=1}^d \frac{v_l}{2}(z^l + z^{-l}). 
\end{align}
{\fab Here, for \( l = 1, \ldots, d \), we have \( v_l \in \mathbb{R} \), and the degree of \( \mathcal{V} \) is \( d \) provided that \( v_d \neq 0 \).}
For such potential, the equilibrium measure is supported by $\T$ or assumption (A2') holds ({\fab see} \cite[Corollary 1.2]{pritsker2005}).  

{\fab Let us briefly discuss polynomial potentials. Naturally, the behavior differs depending on whether the support is full, one-cut, or multi-cut. First, the equilibrium measure is often given explicitly by a trigonometric polynomial, as stated in the following lemma.}

\begin{lemma}\cite[Problem 16.3.1]{pastur_shcherbina}\footnote{There is a typo in Problem 16.3.1 but not in Problem 16.3.2}.
\label{lemps}
{\fab If \( \mathcal{V} \) is a polynomial of the form given in \eqref{even}, then the equilibrium measure \( \mu_{\mathcal{V}} \) admits a Lebesgue density given by
\begin{align}
\rho_{\mathcal{V}}(e^{i\theta}) = \frac{1}{2\pi} \left(1 - \sum_{l=1}^M l v_l \cos(l\theta) \right),
\end{align}
provided that \( \rho_{\mathcal{V}} \) is nonnegative on \( \mathbb{T} \); this holds in particular if \( \sum_{l=1}^M l |v_l| \leq 1 \).}
\end{lemma}
{\fab Conversely, and more generally, consider densities proportional to the function}
\begin{align}
\label{general_form}
H(e^{i\theta}) = \prod_{j=1}^k [1 - \cos(\theta -\theta_j)]^{m_j} = \sum_{l= -M}^M h_l e^{i l\theta}
\end{align}
{\fab for $h_j \in \C$ and  $M=\sum_{j=1}^k m_j$. The corresponding external potentials are then Laurent polynomials with no constant term, given by,}
\[\mathcal V(z) = \sum_{l=1}^M \left(\frac{h_l}{l} z^l + \frac{h_{-l}}{l}z^{-l}\right)\,.\]
{\fab In \cite{breuersizei2018}, the authors consider three specific ensembles with $k = 2$ in equation~(\ref{general_form}), where both roots are equal ($\theta_1 = \theta_2 = 0$). These examples are summarized in Table \ref{tab:3examples}.}

\begin{table}[htbp]
\centering
\begin{tabular}{@{}lll@{}}
\toprule
\( (m, n) \) & \( \rho^{m,n}
\) & \( \mathcal{V}^{m,n}
\) \\
\midrule
\( (1,0) \) & \( \dfrac{1}{2\pi}(1 - \cos\theta) \) & \( \cos\theta\)\\ 
\\[-0.8em]
\( (1,1) \) & 
\(
\begin{aligned}
&\dfrac{1}{\pi}(1 - \cos\theta)(1 + \cos\theta)
\end{aligned}
\) & \( \dfrac{1}{2}\cos 2\theta \)\\
\\[-0.8em]
\( (2,0) \) & 
\(
\begin{aligned}
&\dfrac{1}{3\pi}(1 - \cos\theta)^2 
\end{aligned}
\) & 
\(
\begin{aligned}
& -\dfrac{1}{6} \cos 2\theta + \dfrac{4}{3} \cos \theta 
\end{aligned}
\) \\
\bottomrule
\\
\end{tabular}
\caption{{\fab Expressions for \( \rho^{m,n} \) and their corresponding potentials \( \mathcal{V}^{m,n} \)}}
\label{tab:3examples}
\end{table}

\section{The workshop of Breuer-Simon-Zeitouni}
\label{sec:BSZworkshop}

{\fab A general method for addressing large deviation principles with the goal of deriving sum rules is presented in \cite{breuersizei2018}. The results are established for polynomial potentials whose associated equilibrium measures have full support. 
An additional assumption is imposed on $\rho_\V$
introduced solely to simplify the computational analysis. 
For the sake of completeness, we summarize in this section the main results obtained through this method (Theorems \ref{BSZabstract} and \ref{fullssr} below). We highlight one of its key components (namely, the decomposition of \( \operatorname{tr}\, \mathcal{V}(\mathcal{G}_n(\alpha)) \), recalled in Proposition \ref{propclue}). The section concludes with a summary of the sum rules corresponding to the cases listed in Table \ref{tab:3examples}.}


\subsection{General results}

{\fab One of the  key tool to handle polynomial potential is the following proposition that leads to a decomposition of the density \eqref{krishnaT}.} The general formula in \eqref{227} is taken from \cite[Theorem 3.2]{breuersizei2018}, for the representation of the function $G$ we used a later development in \cite{yan2018}, see the proof of Lemma 2.1 on p. 461 therein. 
\medskip

\begin{proposition}
\label{propclue}
{\fab Let \( \mathcal{V} \) be a Laurent polynomial of degree \( d \). Then, there exist polynomials \( F_{\pm} \) and \( G \), independent of \( n \), such that:
\begin{itemize}
    \item \( G \) depends on \( d+1 \) consecutive values of \( \alpha_j \) and \( \bar{\alpha}_j \),
    \item \( F_{\pm} \) depend on \( d \) such variables,
\end{itemize}
and for all \( n \geq d+1 \), the following holds,
\begin{align}
\label{227}
\tr\!\ \V\left(\mathcal G_n (\alpha)\right) =  F_- (\alpha_{[d-1]}) + F_+(\alpha_{[n-d, n-1]})
+ \sum_{j=0}^{n-1-d} G(\alpha_{[j, j+d]})\,.
\end{align}
More precisely,}
\begin{align}
\label{defGK}
 G(x_{[d]})&= \sum_{k=1}^d \Gamma_{k}(x_{[d]})\\
 \Gamma_{k}(x_{[d]}) &= \sum _{0 \leq j_1\leq j_2\leq \dots \leq j_{2k}\leq k}c_{j_1, \dots, j_{2k}} x_{j_1}\bar x_{j_2}\dots x_{j_{2k-1}}\bar x_{j_{2k}}\,,
 \end{align}
 for some real coefficients $c_{j_1, \dots, j_{2k}}$.
\end{proposition}

As a consequence, the authors of \cite{breuersizei2018} derive the following \emph{abstract gem} and sum rule in the no-cut case.

\begin{theorem}[{\cite[Theorem 3.5]{breuersizei2018}}]
\label{BSZabstract}
If $\V$ is a symmetric Laurent polynomial such that the support of $\mu_\V$ is $\mathbb T$, then
\begin{align} \label{eq:BSZabstract}
\lim_{L\to \infty}\sum_{j=0}^L \left[G(\alpha_{[j, j+d]}) - \log (1- |\alpha_j|^2)\right]
\end{align}
exists and is finite if and only if $\mathcal K(\mu_\V|\mu)$ is finite.
\end{theorem}

\begin{theorem}[{\cite[Theorem 3.6]{breuersizei2018}}]
\label{fullssr}
If $\V$ is a symmetric Laurent polynomial such that the support of $\mu_\V$ is $\mathbb T$, then for all measures $\mu$ on $\mathbb T$
\begin{align}
\notag
\mathcal K(\mu_\V | \mu) =& F_-(\alpha_{[d-1]}) + \lim_{L\to \infty}\sum_{j=0}^L \left[G(\alpha_{[j, j+d]}) - \log (1- |\alpha_j|^2)\right]\\
\label{231}
&-F_-(\alpha_{[d-1]}^\V) - \sum_{j=0}^\infty \left[G(\alpha^\V_{[j, j+d]}) - \log (1- |\alpha^\V_j|^2)\right]
\,.
\end{align}
\end{theorem}
As in the Verblunsky-Szeg{\H o} theorem, the finiteness of $\mathcal K(\mu_\V | \mu)$ implies that $\mu$ is supported by $\T$ and then $\alpha_k\to 0$ as a consequence of Rakhmanov's theorem (see Theorem \ref{Rakhm}). Therefore, if \eqref{eq:BSZabstract} or the right hand side of \eqref{231} are finite, we have $\alpha_k\to 0$. This allows to rewrite those formulas so that the decomposition of $\tr\!\ \V\left(\mathcal G_n (\alpha)\right)$ as in Proposition \ref{propclue} is not necessary. Using the boundedness of V-coefficients, 
\begin{align} \label{decompositionerror}
\left| \left( F_-(\alpha_{[d-1]}) + \sum_{j=0}^L G(\alpha_{[j, j+d]}) \right)- \tr\!\ \V\left(\mathcal G_L (\alpha)\right)  \right| \leq | F_+(\alpha_{[n-d, n-1]}) | \leq C
\end{align} 
for a constant $C$ depending only on $\V$. In the situation of Theorem \ref{BSZabstract}, this implies that $\mathcal K(\mu_\V | \mu)$ is finite if and only if 
\begin{align} \label{eq:BSZabstract-2}
\lim_{L\to\infty} \left( \tr\!\ \V\left(\mathcal G_L (\alpha)\right) - \sum_{j=0}^L \log (1- |\alpha_j|^2) \right)
\end{align}
exists and is finite. Moreover, the sum rule in Theorem \ref{fullssr} may be rewritten as 
\begin{align}\label{231-2}
\mathcal K(\mu_\V | \mu) = 
\lim_{L\to\infty} \left( \tr\!\ \V\left(\mathcal G_L (\alpha)\right) - \tr\!\ \V\left(\mathcal G_L (\alpha^\V)\right) - \sum_{j=0}^L \log \left( \frac{1- |\alpha_j|^2}{1-|\alpha_j^\V|^2}\right) \right) .
\end{align}
For our new gems and sum rules in Section \ref{sec:our_results}, we will use this formulation, although the decomposition in Proposition \ref{propclue} is still crucial in the proofs.

\subsection{The three basic examples}
\label{sec:basicexamples}
In this section, we summarize the sum rules and the associated gems computed in \cite{du2023} for the examples listed in Table \ref{tab:3examples}. The corresponding gems are presented in Table \ref{tab:gemtable} below, along with references to the equations in the Appendix where the respective sum rules are stated.
%
%
 \medskip
\begin{table}[htbp]
\centering
\begin{tabular}{@{}lcc@{}}
\toprule
\( (m, n) \) & Gems
& Sum Rule
 \\
\midrule
\( (1,0) \) & 
$\mathcal K(\mu^{1,0}| \mu) < \infty \Longleftrightarrow \sum_{k=1}^\infty |\alpha_{k+1} - \alpha_{k}|^2 + |\alpha_k|^4 < \infty\,$ & (\ref{sr10})\\ 
\\[-0.8em]
\( (1,1) \) & 
$\mathcal K(\mu^{1,1}| \mu) < \infty \Longleftrightarrow \sum_{k=1}^\infty |\alpha_{k+2} - \alpha_{k}|^2 + |\alpha_k|^4 < \infty\,$ & (\ref{sr11})\\
\\[-0.8em]
\( (2,0) \) & 
$\mathcal K(\mu^{2,0}|\mu) <\infty \ \Longleftrightarrow \ \sum_{k=0}^\infty\left(|\alpha_{k+2}-2\alpha_{k+1} +\alpha_k|^2 + |\alpha_k|^6\right) <\infty\,$ & 
(\ref{sr20}) \\
\bottomrule
\\
\end{tabular}
\caption{{\fab Gems and sum rules for the potentials \( \mathcal{V}^{m,n} \)}}
\label{tab:gemtable}
\end{table}

\section{\fab Main results. Sum rules and gems for proper arcs}
\label{sec:our_results}

We now present our new sum rules and gems for the general polynomial case. Unlike in Section \ref{sec:BSZworkshop}, we assume that the support of $\mu_\V$ is not the full circle, that is, we are in the case of assumption (A2'), or perhaps under the stronger assumption (A2). In this more general case, the equilibrium measure $\mu_\V$ has support $I=\operatorname{supp}(\mu_\V)$, which is the union $I = I_1\cup \dots \cup I_m$ of $m$ disjoint arcs, each with nonempty interior. We introduce the set $\Sr_1(I)$ of probability measures $\mu$ on $\mathbb T$ with 
\begin{align} \label{support}
\operatorname{supp}(\mu) = J \cup E ,
\end{align}
where $J\subset I$ and $E=E(\mu)$ a finite or countable subset of $\mathbb{T}\setminus I$.  

Given a polynomial $\V$, we define 
\begin{align} \label{defRorW}
{\mathcal R_L }(\alpha_{[L-1]}) :=  \tr\!\ \V\left(\mathcal G_L (\alpha)\right) - \tr\!\ \V\left(\mathcal G_L (\alpha^\V)\right) - \sum_{j=0}^L \log \left( \frac{1- |\alpha_j|^2}{1-|\alpha_j^\V|^2}\right) .
\end{align}
%
{\fab Our main results are the two following theorems.}

\begin{theorem}[Abstract gem] \label{Unitabstractgem}
Suppose that $\V$ is a symmetric Laurent polynomial satisfying assumption (A2'). Then{\fab ,} for {\fab any} probability measure $\mu$ on $\T$ with infinite support and $\alpha$ the sequence of its $V$-coefficients, we have 
 \begin{align}
\label{Unitgengem}
\sup_{L\geq 1} {\mathcal R_L }(\alpha_{[L-1]}) < \infty 
\end{align}
if and only if the following conditions are satisfied:
\begin{enumerate}
\item 
  $\mu \in \mathcal S_1(I)$, 
\item when $\ddd\mu(z) = w(z) \ddd\mu_\V(z)+ \ddd\mu_s(z)$ 
is the Lebesgue decomposition with respect to $\mu_\V$, 
\begin{align*}
\int \log w(z) \ddd\mu_\V(z) >-\infty\,,
\end{align*}
\item 
 $\displaystyle \sum_{\lambda \in E(\mu)} \mathcal{F}_\V(\lambda)<\infty$ .
\end{enumerate}
\end{theorem}
{\fab A natural problem is to determine sufficient conditions on the measure in order that the limit of 
\(\mathcal{R}_L(\alpha_{[L-1]})\) as \(L \to \infty\) exists so that a sum rule can be established.}
%
%
{\fab In the next theorem, we  give a partial answer to this question and get a sum rule under a condition on the V-coefficients.}
For $\mu_\V$ an equilibrium measure with V-coefficients $\alpha^\V$, we let 
\begin{align}\label{defMclass}
M_\V = \left\{\mu \in \mathcal M_1(\T): \lim_{k\to \infty} |\alpha_k-\alpha^\V_k| = 0 \right\} . 
\end{align}
\nomenclature{$M_\V$}{$\{\mu \in \mathcal M_1(\T): \lim_{k\to \infty} \NomBar\alpha_k-\alpha^\V_k\NomBar = 0\}$}

\begin{theorem}
\label{Unitabstractnewsumrule}
Let $\V$ be a symmetric Laurent polynomial satisfying assumption (A2') and $\mu \in \mathcal S_1(I)$ a probability measure with infinite support. If $\mu \in M_\V$, then the limit $\lim_{L\to \infty} \mathcal R_L( \alpha_{[L-1]})\in (-\infty,\infty]$ exists and  
\begin{align*}
\mathcal K(\mu_\V \mid \mu) + \sum_{x \in E(\mu)} \mathcal F_\V(x) = 
 \lim_{L\to \infty} \mathcal R_L( \alpha_{[L-1]}) . 
\end{align*}
\end{theorem}

In the one-cut case, we can observe a connection between the convergence assumption on the coefficients and finding a specific representative in a class of Aleksandrov-Clark measures {\fab(see \cite{simonopuc1} p.35)}. 
For $\mu$ a measure with V-coefficients $\alpha_k, k\geq 0$ and $\lambda \in \T$, let $\mu_\lambda$ be the measure with coefficients $\lambda \alpha_k$. Then{\fab ,} the set 
\begin{align}\label{aleksandrov}
\{\mu_\lambda: \lambda \in \T\}
\end{align}
is the set of Aleksandrov measures associated to the Carath\'{e}odory function $F$ of $\mu$, 
\begin{align}
F(z) = \int \frac{\zeta+z}{\zeta-z}\ddd \mu(\zeta) , 
\end{align} 
see \cite{simonopuc1}, Section 1.3 and Section 3.2. More precisely, the Carath\'{e}odory function $F_\lambda$ of $\mu_\lambda$ is related to $F$ via
\begin{align*}
F_\lambda(z) = \frac{(1-\lambda)+(1+\lambda)F(z)}{(1+\lambda)+(1-\lambda)F(z)} . 
\end{align*}
As proven in \cite[Theorem 3.2.16]{simonopuc1}, all measures $\mu_\lambda$ have the same essential support.

\begin{corollary}
\label{Unitnewsumrulesymmetric}
Let $\V$ be a symmetric Laurent polynomial satisfying assumption (A2) and $\mu \in \mathcal S_1(I)$ a symmetric probability measure with infinite support. Then{\fab ,} for $\lambda=1$ or for $\lambda=-1$, the limit $\lim_{L\to \infty} \mathcal R_L(\lambda \alpha_{[L-1]})\in (-\infty,\infty]$ exists and 
\begin{align} \label{corollarysumrule}
\mathcal K(\mu_\V \mid \mu_\lambda) + \sum_{x \in E(\mu_\lambda)} \mathcal F_\V(x) = 
 \lim_{L\to \infty} {\mathcal R_L}(\lambda\alpha_{[L-1]}) .
\end{align}
\end{corollary}

%

\begin{remark}
\label{rem:counterexample}
In general, it is not possible to extend the identity in Theorem \ref{Unitabstractnewsumrule} to all measures $\mu\in \mathcal S_1(I)$. We can give a counterexample already in the one-cut case as in Corollary \ref{Unitnewsumrulesymmetric}, for the details see Section \ref{app:remark}. For this we consider as reference measure the Gross-Witten measure $\GW_\g$, the equilibrium measure corresponding to the potential $\V(z) =  \g \mathrm{Re}(z)$. For $\g > 1$ (gapped case), the Gross-Witten measure is supported by an arc $[\pi-\theta_\g, \pi+\theta_\g]$, with $\theta_{g}  \in (0, \pi)$, see Section \ref{susec:GW}. 
The V-coefficients $\alpha_n^\g$ of $\GW_\g$ are explicitly known and satisfy 
\begin{align}
\label{limangle}
\lim_{n\to \infty} \alpha_n^\g = - \sqrt{1- |\g|^{-1}}=: \mathfrak a .
\end{align} 
For this reference measure, let us suppose the identity of Theorem \ref{Unitabstractnewsumrule} holds for all symmetric $\mu\in \mathcal{S}_1(I)$ and write it as 
\begin{align} \label{suspectedidentity}
\mathrm{LHS}(\mu) = \mathrm{RHS}(\mu). 
\end{align}
We compare the two sides for two particular measures $\mu$ with constant V-coefficients, related to the Geronimus polynomials. Specifically, following \cite{gamboanag2017}, we consider the Hua-Pickrell distribution $\GE_{\mathfrak a}$, defined as the symmetric measure with constant V-coefficients $\alpha_n( \GE_{\mathfrak a})=\mathfrak a$. 
Then{\fab ,} $\GE_{\mathfrak a}$ has the same support as $\GW_\g$ and $\GE_{\mathfrak a} \in M_\V$ (see \eqref{eq:geronimusmeasure}, setting $-\gamma=\mathfrak a$). By Corollary \ref{Unitnewsumrulesymmetric},
\begin{align*}
\mathrm{LHS}(\GE_{\mathfrak a}) = \mathrm{RHS}(\GE_{\mathfrak a}) .
\end{align*}
One may then calculate that 
\begin{align*}
\mathrm{LHS}(\GE_{-\mathfrak a}) - \mathrm{LHS}(\GE_{\mathfrak a}) = \log \left(\frac{1-\mathfrak a}{1+\mathfrak a}\right) + \frac{1}{2}\left(\frac{\sqrt{|\mathfrak a|}}{1+\mathfrak a}- \log \frac{1+\sqrt{|\mathfrak a|}}{\sqrt{1+\mathfrak a}}\right). 
\end{align*}
On the other hand, $\mathrm{RHS}(\GE_{-\mathfrak a}) - \mathrm{RHS}(\GE_{\mathfrak a})$ is a polynomial in $\alpha_k^\V$ and $\mathfrak a$. Therefore, identity \eqref{suspectedidentity} cannot be true for $\mu = \GE_{-\mathfrak a}$. This argument also shows that there is no polynomial modification of the right hand side, such that \eqref{suspectedidentity} holds for all $\mu \in \mathcal S_1(I)$.    
\end{remark}
\begin{remark}
\label{remconjecture}
Corollary \ref{Unitnewsumrulesymmetric} also raises the question whether the statement may holds for some nonsymmetric measures. 
One way to answer the question would be  to 
consider Aleksandrov class of a  symmetric probability measure and to show that we can find a $\lambda\in \T$, such that $\mu_\lambda\in M_\V$,
(so Theorem \ref{Unitabstractnewsumrule} can be applied). However, such a $\lambda$ might not exist.  As an example, for  $\V=0$ when $\mu_\V=\UNIF$ with full support, Simon in \cite[Theorem 12.6.2]{simonopuc2} gives examples of i.i.d. V-coefficients, such that the measure $\mu$ has a pure point spectrum dense on $\mathbb T$ and  $\alpha_k$ does not converge to 0. In this case, $\mu \notin M_0$, and in fact $\mu_\lambda \notin M_0$ for any $\lambda \in \mathbb T$.    Therefore, the more promising approach to extend Corollary \ref{Unitnewsumrulesymmetric} would be to show that for $\mu\in \mathcal S_1(I)$ such that the left hand side of \eqref{corollarysumrule} is finite, we  can find a $\lambda\in \T$, such that $\mu_\lambda\in M_\V$. 
\end{remark}

\begin{remark}
For probability measures on the real line, a similar phenomenon can be observed: in \cite{gamboanag2021}, a general abstract gem for polynomial potentials was proven and a sum rule in the one-cut case. This sum rule holds without additional assumptions on the coefficients. For reference measures in the multi-cut case, the sum rule does not hold in general.  It seems surprising that the one-cut case on the unit circle needs additional constraints. With the Szeg\H{o} mapping, there is a straighforward way to map a symmetric probability measure $\mu$ on the circle to a measure $\nu$ on $[-2,2]$, with an explicit relation between their recursion coefficients (the Geronimus relation). As explored in \cite{gamboanag2023}, this allows to derive sum rules on the unit circle from those on the real line and vice-versa. It would be possible to formulate a sum rule for symmetric measures on the circle in the one-cut case without an additional constraint, but then the right hand side is parametrized in terms of the recursion coefficients of the measure $\nu$, and not in terms of the V-coefficients of $\mu$. 
\end{remark}
\section{One-parameter family of potentials - Ungapped case}
\label{sec:one-par-ungapped}

We now revisit the examples of Section \ref{sec:basicexamples}, introducing a multiplicative parameter $\g\in \mathbb R$ in front of their potential $\V$ and study the family of potentials $\V_\g = \g\V$.
The parameter $\g$ acts as  an inverse temperature and  measures the strength of the coupling.
For low values of $|\g|$ (strong coupling)  the model is ungapped, with $\mu_\g=\mu_{\V_\g}$ supported by $\T$, for other values it is one-cut or multi-cut. We begin by stating the sum rule in the ungapped case, as it follows from the result of Breuer, Simon, and Zeitouni \cite{breuersizei2018}. In Section \ref{sec:one-par-gapped}, we present new sum rules in the one-cut or multi-cut situation.    

The right hand side is easily modified by multiplying the potential by $\g$ and we have sum rules of the form \eqref{231-2}: 
\begin{align*}
\mathcal K(\mu_\g | \mu) = 
\lim_{L\to\infty} \left( \g\, \tr\!\ \V\left(\mathcal G_L (\alpha)\right) - \g\, \tr\!\ \V\left(\mathcal G_L (\alpha^\V)\right) - \sum_{j=0}^L \log \left( \frac{1- |\alpha_j|^2}{1-|\alpha_j^\V|^2}\right) \right) .
\end{align*}
Additionally, we may simplify the right hand side by evaluating the contribution of $\alpha^\V$. Since the V-coefficients of the $\g$-equilibrium measure are in general unknown, we use $\mu=\UNIF$ as a test measure. For this measures with V-coefficients all zero (with a possible contribution of $\alpha_{-1}=-1$) and in the examples in this section $\mathcal K (\mu_\g|\UNIF)<\infty$. This leads to  
\begin{align}\label{231-3}
\mathcal K(\mu_\g | \mu) = \mathcal K(\mu_\g | \UNIF) + 
\lim_{L\to\infty} \left( \g\, \tr\!\ \V\left(\mathcal G_L (\alpha)\right) - \g\, \tr\!\ \V\left(\mathcal G_L (0)\right) - \sum_{j=0}^L \log \left( 1- |\alpha_j|^2\right) \right) .
\end{align}
In this setting, we now present several sum rules for the one-parameter family of potentials, using the source potentials $\V^{1,0},\; \V^{1,1},\; \V^{2,0}$. We begin by giving in the next table the potentials and the corresponding full support density of the equilibrium measures $\mu_{g}^{m,n}$. Here, we take $0\leq\g<1$.
\begin{table}[htbp]
\centering
\begin{tabular}{@{}lll@{}}
\toprule
\( (m, n) \) & \( \mathcal{V}^{m,n}_\g \)
& \( \rho_\g^{m,n}
\) \\
\midrule
\( (1,0) \) & \( \g \cos\theta\) & \( \frac{1}{2\pi}(1-\g\cos\theta)\)\\ 
\\[-0.8em]
\( (1,1) \) & 
\(
\begin{aligned}
&\dfrac{\g}{2}\cos 2\theta
\end{aligned}
\) & \( \frac{1}{2\pi}(1-\g\cos2\theta)\)\\
\\[-0.8em]
\( (2,0) \) & 
\(
\begin{aligned}
&\dfrac{4\g}{3}\cos\theta-\dfrac{\g}{6}\cos 2\theta 
\end{aligned}
\) & 
\(
\begin{aligned}
& \dfrac{1}{2\pi}\left(1-\dfrac{4\g}{3}\cos\theta + \dfrac{\g}{3} \cos 2\theta\right) 
\end{aligned}
\) \\
\bottomrule
\\
\end{tabular}
\caption{Expressions for \( \mathcal{V}^{m,n}_\g \) and the corresponding equilibrium densities}
\label{tab:oneparameter}
\end{table}
\\
Recall that the family generated in the $(1,0)$ case is also known as the Gross-Witten ensemble of parameter $\g$ ($\GW_\g$).
Define,
$$
H(\g) =  \mathcal K(\mu_{g}^{1,0}| \UNIF)
= 1- \sqrt{1- \g^2} + \log \frac{1 + \sqrt{1-\g^2}}{2}.
$$
Define further $K(g)$ in (\ref{522'}) and $I(\g)$ in  (\ref{Ig}) with $\alpha$ in (\ref{defalpha}) and $a$ in (\ref{defha}).
We then have the following sum rules.
\begin{theorem}
\label{th:ungap}
For $\g\in [0,1)$, $(m,n)=(1,0),\;(1,1)\;(2,0)$  and $\mu\in \mathcal M_1(\T)$ nontrivial,
\begin{align}
\mathcal K(
\mu_{g}^{m,n} |  \mu)  < \infty \  \Leftrightarrow \ 
\alpha \in \ell^2 .
\label{gem2ungap}
\end{align}
More precisely, we have 
\begin{equation}
\label{srungap}
\mathcal K(\mu_{g}^{m,n} | \mu)=A_g^{m,n}.
\end{equation}
Where
\begin{eqnarray*}
A_g^{1,0}&=& H(\g)  -\frac{\g }{2} + \frac{\g }{2} \sum_{k=0}^\infty|\alpha_k - \alpha_{k-1}|^2  
+ \sum_{k=0}^\infty \left(- \log (1 - |\alpha_k|^2)  - \g|\alpha_k|^2\right),\\
A_g^{1,1} &= & H(\g) -\frac{3}{4}\g+2\sqrt{1-\g^2}-2 \notag \\
& &\quad +\sum_{k=0}^\infty\left( -\log(1-|\alpha_k|^2) 
- \g  |\alpha_k|^2 -\frac{ \g }{2}|\alpha_k|^4\right)\notag \\
 & &\quad +  \g \sum_{k=0}^\infty |\alpha_k\alpha_{k-1}|^2 + \frac{ \g }{2}\sum_{k=0}^\infty (1-|\alpha_k|^2) |\alpha_{k+1}-\alpha_{k-1}|^2
\\
\notag
& &\quad +\frac{ \g }{8} \sum_{k=0}^\infty \left(\left(2|\alpha_k|^2 - |\alpha_k - \alpha_{k-1}|^2\right)^2 + \left(2|\alpha_{k-1}|^2 - |\alpha_k - \alpha_{k-1}|^2\right)^2\right),\\
A_g^{2,0} &= & \frac{19\g }{12} + K(\g) - \g I(\g) \\
\notag
& &
+\sum_{k=0}^\infty\left( -\log(1-|\alpha_k|^2) -\g |\alpha_k|^2 -\frac{\g}{2}|\alpha_k|^4\right)\\
 \notag
& &+\frac{\g}{6}\sum_{k=0}^\infty \left(|\alpha_k|^2 - |\alpha_{k-1}|^2\right)^2\\
 \notag
 & &+\frac{\g}{6}\sum_{k=0}^\infty (1- |\alpha_k|^2) |\alpha_{k+1}-2\alpha_k + \alpha_{k-1}|^2\\
 & &+ \frac{\g}{12}\sum_{k=0}^\infty\left(6 |\alpha_k|^2 + 
 6 |\alpha_{k-1}|^2-|\alpha_k -\alpha_{k-1}|^2\right)|\alpha_k -\alpha_{k-1}|^2.
\end{eqnarray*}
\end{theorem}
\noindent
The proofs of these sum rules are postponed to Appendix \ref{append:ungaped}.
%
%

%
\section{One-parameter family of potentials - Gapped phase}
\label{sec:one-par-gapped}

\subsection{The (1,0) case: Gross-Witten ensemble}
\label{susec:GW}

We return to the
Gross-Witten ensemble, now in the gapped or weakly coupled phase $|\g| > 1$. 
The expression of the equilibrium measure $\GW_\g$ is different according to the sign of $\g$:
\begin{align}
\label{GWeq}
\rho_\g^{1,0}(e^{i\theta}) = \begin{cases}
\displaystyle 
\frac{|\g|}{\pi}\cos(\tfrac{\theta}{2})\!\  \sqrt{\sin^2(\tfrac{\theta_\g}{2})-\sin^2(\tfrac{\theta}{2})}
\mathbbm{1}_{[-\theta_\g, \theta_\g]}(\theta)  &\hbox{if} \ \g < -1\,,\\
\displaystyle \frac{\g}{\pi}\sin(\tfrac{\theta}{2})\!\  \sqrt{\sin^2(\tfrac{\theta}{2})-\cos^2(\tfrac{\theta_\g}{2})}\mathbbm{1}_{ [\pi-\theta_\g, \pi+\theta_\g]}(\theta) 
  &\hbox{if}\ \g > 1\,,
\end{cases}
\end{align}
where $\theta_{g}  \in [0, \pi]$ is such that 
\begin{equation}\label{eqthetag}
\sin^2 (\tfrac{\theta_\g}{2}) = |\g|^{-1}\,.
\end{equation}
We will treat only the case $\g >1$, and in this case the support is 
\begin{align} \label{circlearc}
\{z=e^{i \theta} \in \mathbb{T}\!\ |\, \theta \in [\pi-\theta_\g, \pi +\theta_\g]\}
\end{align}
 with $\theta_\g$ as in \eqref{eqthetag}. The function $\mathcal F_\V$ as defined in \eqref{defeffectivepot} is given by  
\begin{align*}
\mathcal F_{GW_\g}(e^{i\theta}) =  
 2|\g|\displaystyle\int_{I(\theta)}\sin(\varphi/2)\sqrt{ \cos^2 (\theta_\g/2)  - \sin^2(\varphi/2)}\ d\varphi \,,
 \end{align*}
 where 
 \begin{align*}
 I(\theta) = \begin{cases}[\theta, \pi-\theta_\g] & \hbox{if} \ \theta \in (0, \pi-\theta_\g)\,,\\
 [\pi + \theta_\g, \theta]& \hbox{if} \ \theta \in (\pi+\theta_g , 2\pi)\,.
\end{cases}
\end{align*}

From \cite[formula (7.22)]{Zhedanov1998}, we know that the V-coefficients of $\GW_\g$ are
\begin{align}\label{alphalimGW+}
\alpha_{n-1}^\g 
= 1 - \frac{2}{1+q}\frac{1 - q^{n+2}}{1 - q^{n+1}}\,,
\end{align}
where $q = \left(\sqrt{\g} - \sqrt{\g  -1}\right)^2$. Since $0 < q<1$, we have
\begin{align}
\label{limangle1}
\lim_{n\to \infty} \alpha_n^\g = 1 - \frac{2}{1+q}  = - \sqrt{1- \g^{-1}} = -\cos(\theta_\g/2) =: \mathfrak a \,.
\end{align}

Theorem \ref{Unitabstractgem} implies the following new abstract gem.

\begin{theorem}
\label{newGW}
 Let $\g> 1$ and let $\mu\in \mathcal{M}_1(\mathbb{T})$ be nontrivial. Then{\fab ,}
\begin{equation}
\label{newnewGW} 
\sup_{L\geq 1} \sum_{k=0}^L\left(-\g \mathrm{Re} (\alpha_k\bar\alpha_{k-1} - \alpha_k^\g\bar\alpha_{k-1}^\g)  \
- \log \frac{1- |\alpha_k|^2}{1-|\alpha_k^\g|^2}\right) <\infty
\end{equation}
if and only if $\mu \in \mathcal{S}^{\mathbb{T}}_1(\pi-\theta_\g,\pi+\theta_\g)$ and
\begin{align*}
\mathcal K(\GW(\g) \mid  \mu) + \sum_{\lambda\in E(\mu)} \mathcal F_{\GW(\g)}(\lambda) < \infty .
\end{align*} 
\end{theorem}

As a consequence, we can formulate a proper gem in the following proposition.

\begin{proposition}
\label{necandsuff}
Let $\mathfrak {a} = -\sqrt{1- \g^{-1}}$ (see \eqref{limangle}).
The supremum in (\ref{newnewGW}) is finite if and only if
 the two following conditions are satisfied:
\begin{align}
\label{C1}
\sum_{k=0}^\infty |\alpha_{k+1} - \alpha_k|^2 < \infty\\
\label{C2}
\sum_{k=0}^\infty (|\alpha_k| - \mathfrak{a})^2 < \infty\,.
\end{align}
\end{proposition}

\noindent\textbf{Proof of Proposition \ref{necandsuff}:}
Let us reduce the problem of convergence of the series to the study of some simpler series.
We have  for $x,y \in \C$
\[2 \mathrm{Re}(x\bar y) = |x|^2+ |y|^2 - |x-y|^2\]
so that 
\begin{align*}
2\mathrm{Re} (\alpha_k\bar\alpha_{k-1} - \alpha_k^\g\bar\alpha_{k-1}^\g) =
(|\alpha_k|^2 -\mathfrak{a}^2)+ (|\alpha_{k-1}|^2-\mathfrak{a}^2) - |\alpha_{k-1} - \alpha_k|^2 + R_k^\g
\end{align*}
where
\begin{align*}
R_k^\g =  |\alpha_{k-1}^\g - \alpha_k^\g|^2- (|\alpha_k^\g|^2 -\mathfrak{a}^2)- (|\alpha_{k-1}^\g|^2-\mathfrak{a}^2)
\end{align*}
is the general term of a convergent series.
We  also have
\begin{align*}
\log \frac{1- |\alpha_k|^2}{1-|\alpha_k^\g|^2} = \log \frac{1- |\alpha_k|^2}{1-\mathfrak{a}^2} - \log \frac{1- |\alpha_k^\g|^2}{1-\mathfrak{a}^2}\,,
\end{align*}
and the last term is the general term of a convergent series. 
Thus, \eqref{newnewGW} holds if and only if 
\begin{align*}
S_L = \sum_{k=0}^L \frac{\g}{2} |\alpha_{k-1} - \alpha_k|^2- \g (|\alpha_k|^2 -\mathfrak{a}^2)  
-\log \frac{1- |\alpha_k|^2}{1-\mathfrak{a}^2} 
\end{align*}
remains bounded. 
Since 
\begin{align*}
\frac{1- |\alpha_k|^2}{1-\mathfrak{a}^2}  &= 1 - \frac{|\alpha_k|^2 - \mathfrak{a}^2}{1-\mathfrak{a}^2}, 
\end{align*}
we have the expansion
\begin{align*}
-\log \frac{1- |\alpha_k|^2}{1-\mathfrak{a}^2}  =   \frac{|\alpha_k|^2 - \mathfrak{a}^2}{1-\mathfrak{a}^2} + 
\frac{1}{2} \frac{(|\alpha_k|^2 - \mathfrak{a}^2)^2}{(1-\mathfrak{a}^2)^2} +
O\left((|\alpha_k|^2 - \mathfrak{a}^2)^3\right) .
\end{align*}
Using $(1-\mathfrak{a}^2)^{-1} = \g$, we see that
\begin{align*}
 S_L &= 
\label{decOmp}
  \sum_{k=0}^{L} 
  \frac{\g}{2}  |\alpha_{k-1} - \alpha_k|^2
+\frac{\g^2}{2} (|\alpha_k|^2 -\mathfrak{a}^2)^2 +
O\left((|\alpha_k|^2 - \mathfrak{a}^2)^3\right) .
\end{align*} 
The double inequality
\begin{align*}
|\mathfrak{a}|^2 
 \leq\frac{ (|\alpha_k|^2 - \mathfrak{a}^2)^2} { \left(|\alpha_k|-\mathfrak{a}\right)^2}
\leq (1 + |\mathfrak{a}|^2) 
\end{align*}
entails $S_L$ is bounded if and only if  
\begin{align*}
\tilde S_L = \sum_{k=0}^L \g |\alpha_{k-1} - \alpha_k|^2 +  \g^2(|\alpha_k| - \mathfrak{a})^2  
\end{align*}
is bounded. Since $\g>1$, this is clearly equivalent to \eqref{C1} and \eqref{C2}
\qed

\begin{remark}
During the proof of the above proposition we obtained as an intermediate step the conditions
\begin{align*}
\sum_{k=0}^\infty |\alpha_{k+1} - \alpha_k|^2 < \infty, \qquad 
\sum_{k=0}^\infty (|\alpha_k|^2 - \mathfrak{a}^2)^2 < \infty
\end{align*}
and then simplified the second one to \eqref{C2}. In the above formulation, we may set $\g = 1$, so that $\mathfrak{a}=0$ and the second condition becomes $\alpha \in \ell^4$, which is consistent with \eqref{gem10}.
\end{remark}

\subsection{The (1,1) gapped case } 

We consider the potential 
\[-\g\V_{1,1}(e^{i \theta}) = -\frac{\g}{2}\cos 2\theta \] 
for $\g > 1$, where we find a two-cut situation. 
According to \cite[page 22-23]{mandal1990}, the support of the equilibrium measure $\mu_\g^{1,1}$ is 
\begin{align*}
\{ e^{i \theta} \in \mathbb{T}\!\ |\, \theta \in [-\alpha_c, \alpha_c]\cup [\pi -\alpha_c, \pi +\alpha_c] \} , 
\end{align*}
where $ \alpha_c \in (0, \pi/2)$ is the solution of
\begin{align*}
\sin \alpha_c = \g^{-1} \,.
\end{align*}
The density is then{\fab ,}
\begin{align*}
\rho_\g^{1,1}(e^{i\theta}) = \frac{\g}{\pi} |\cos\theta|
\sqrt{\sin^2 \alpha_c -\sin^2 \theta}
\,.
\end{align*}
Notice that when $\g\downarrow 1$, then $\alpha_c \uparrow \pi/2$ and  we recover the density $\rho_{1}^{1,1}$.
\paragraph{$\g< -1$} Since the potential is invariant by the simultaneous change $\g \mapsto - \g\  ,\ \theta \mapsto \frac{\pi}{2} + \theta$,
 we find the support 
$[\gamma_c, \pi- \gamma_c]\cup [\pi+\gamma_c, 2\pi-\gamma_c]$ where $\gamma_c \in [0, \pi/2]$ is solution of
\begin{align*}
\cos \gamma_c =  |\g|^{-1}\,.
\end{align*}
The density is given by
\begin{align*}
\rho^{1,1}_\g(e^{i\theta}) =  \frac{|\g|}{\pi} |\sin\theta|
\sqrt{\cos^2 \gamma_c -\cos^2 \theta}\,.
\end{align*}
Note that when $\g \uparrow -1$, then $\gamma_c\downarrow 0$ and we recover the density 
$\rho^{1,1}_{-1}$.

We may apply Theorem \ref{Unitabstractgem} in this case and obtain that 
\begin{align*}
\sup_{L\geq 1} \mathcal R_L(\alpha_{[L-1]}) < \infty
\end{align*}
if and only if conditions (1), (2) and (3) in Theorem \ref{Unitabstractgem} are satisfied, with reference measure $\mu_\g^{1,1}$ and with potential $\V^{1,1}_\g$.  
The function $\mathcal R_L(\alpha_{[L-1]})$ can be written using \eqref{eq:potentialcomputation11}. Since the formulation is then straightforward, we omit it here. 

One may ask for an equivalent condition for finiteness of $\sup_L\mathcal R_L(\alpha_{[L-1]})$, formulated in terms of $\ell^p$-conditions as in Proposition \ref{necandsuff}. However, it is not obvious whether such an equivalent condition exists. As in the higher order Szeg\H{o} theorems of \cite{du2023}, such an equivalence might only hold under some {\it a priori} condition on the coefficients.   
To investigate these questions using a method similar to that of Proposition \ref{necandsuff}, one intermediate step would be the derivation of asymptotics of the V-coefficients of $\mu_\g^{1,1}$.

\subsection{The (2,0) gapped case}
\medskip

Let us now turn to the potential $\V_\g^{2,0}$ as in (\ref{eq:gpotential20}) 
in the gapped case, i.e. when $\g \in (-\infty,-3/5) \cup (1, \infty)$. The support of the equilibrium measure is then a single arc. The following lemma  is a consequence of computations made in \cite{oota2022} (see also \cite{mandal1990}, \cite{periwal1990}, \cite{oota2022}, \cite{houart1990}). The method is also exposed in \cite{pritsker2005}.

\begin{lemma}
\label{gapg}
\begin{enumerate}
\item When $\g > 1$,  the measure $\mu_\g^{2,0}$  is supported by $\{ e^{i \theta} \in \mathbb{T}\!\ |\, \theta \in [\pi-\theta_c, \pi+ \theta_c]\}$ where $\theta_c$ satisfies
\begin{align*}
 1 - \g^{-1} =\cos^4 (\theta_c/2) \,, 
\end{align*}
and its density is
\begin{align*}
\rho^{2,0}_\g(e^{i\theta}) = \frac{2\g}{3\pi}\sin(\theta/2) \left(1+\cos^2(\theta_c/2) - \cos\theta\right)\sqrt{\sin^2(\theta/2) - \cos^2(\theta_c/2)}\,.
\end{align*}
\item When $\g< -3/5$, the measure $\mu_\g^{2,0}$  is supported by $\{ e^{i \theta} \in \mathbb{T}\!\ |\, \theta \in [-\theta_c, \theta_c]\}$  where $\theta_c$ satisfies
\begin{align*}
\frac{5}{3} +\g^{-1} = \frac{8}{3}\cos^2(\theta_c /2) - \cos^4 (\theta_c/2)\,, 
\end{align*}
and its density  is
\begin{align*}
\rho^{2,0}_\g(e^{i\theta}) =\frac{2|\g|}{3\pi}\cos(\theta/2)  \left(2 + \sin^2(\theta_c/2) - \cos\theta\right)
\sqrt{\sin^2(\theta_c/2) - \sin^2(\theta/2)}\,.
\end{align*}
\end{enumerate}
\end{lemma}

Notice that if $\g \downarrow -3/5$ (resp. $\g \uparrow 1$), then $\theta_c \uparrow \pi$ and we recover $\mu^{2,0}_{-3/5}$ (resp. $\mu^{2,0}_{1}$).

As in the gapped $(1,1)$ case, we may apply Theorem \ref{Unitabstractgem} and obtain an abstract gem with reference measure $\mu_\g^{2,0}$.   
The computation of the function $\mathcal R_L(\alpha_{[L-1]})$ can be made using \eqref{eq:potentialcomputation10} and \eqref{eq:potentialcomputation11}. 

Also in this case, finding $\ell^p$-conditions on the coefficients which are equivalent to $\sup_L\mathcal R_L(\alpha_{[L-1]})<\infty$ is not an easy task. In order to proceed as in Proposition \ref{necandsuff}, one needs to find the asymptotics of the V-coefficients of $\mu_\g^{1,1}$ first, possibly making use of the strategy of \cite{Zhedanov1998}.

\section{Large deviations}
\label{sec:LDP}

In order to be self-contained, we recall the definition of a large deviation principle. Let $(X_n)_n$ be a sequence of random variables with values in some Polish space $\mathcal{X}$ with Borel $\sigma$-algebra. Let $\mathcal{I}:\mathcal{X} \to [0,\infty]$ and $(a_n)_n$ be a sequence of positive real numbers with $a_n\to \infty$. We say that $(X_n)_n$ satisfies the large deviation principle with speed $a_n$ and rate function $\mathcal{I}$, if $\mathcal{I}$ is lower semicontinuous and 
\begin{itemize}
\item[(1)] for all $C\subset \mathcal{X}$ closed
\begin{align*}
\limsup_{n\to \infty} \frac{1}{a_n} \log \mathbb{P}(X_n \in C) \leq - \inf_{x\in C} \mathcal{I}(x) ,
\end{align*} 
\item[(2)] for all $O\subset \mathcal{X}$ open
\begin{align*}
\liminf_{n\to \infty} \frac{1}{a_n} \log \mathbb{P}(X_n \in O) \geq - \inf_{x\in C} \mathcal{I}(x) .
\end{align*} 
\end{itemize}
We will only consider good LDPs, which means that the level sets $\{\mathcal{I}\leq L\}$ of the rate function are compact for all $L\geq 0$.

\subsection{A general spectral LDP}

In \cite{gamboanag2016} and \cite{gamboanag2017}, we proved the LDP for spectral measures in the one-cut case for measure on $\R$ and on $\T$, respectively. In \cite[Theorem 4.3]{gamboanag2021}, we treated the multi-cut case on $\R$. The following theorem is the analogue of the latter result for measures on the unit circle. It is possible to derive it from the case of the real line, see Section \ref{sec:LDPproofspectral}, even with some simplifications since $\T$ is compact.

\begin{theorem}
\label{LDPsp}
Suppose that the potential $\V$ satisfies assumptions (A1), (A2') and (A3). Then the sequence of spectral measures $\mu_n$ satisfies under $\PP_n^\V$ the LDP with speed $n$ and good rate function
\begin{align}
\label{35}
\mathcal I_{\meas}(\mu) = \mathcal K(\mu_\V| \mu) + \sum_{\lambda\in E(\mu)} \mathcal F_\V (\lambda)\,,
\end{align}
if $\mu \in \mathcal S_1(I)$, and $\mathcal I_{\meas}(\mu)=\infty$ otherwise. 
\end{theorem}

\subsection{An abstract coefficient LDP}
\label{susec:abstractcoef}

Recall the homeomorphism $\psi$ as defined in \eqref{defphimapping}, which maps a measure on $\T$ to its finite or infinite sequence of V-coefficients. The continuity of $\psi$ allows to transfer the LDP in Theorem \ref{LDPsp} to the sequence of V-coefficients by the contraction principle. This yields that the coefficients  of $\mu^{(n)}$ satisfy under $\PP_n^\V$ the LDP with speed $n$ and the rate function is 
\begin{align} \label{ratecoeff}
\mathcal{I}_{\coeff} (\alpha) = \inf\{\mathcal I_{\meas}(\mu) : \psi(\mu)=\alpha\} = \mathcal I_{\meas}(\psi^{-1}(\alpha)). 
\end{align}$\alpha^{(n)}$
For fixed $L\geq d+1$, we now consider the sequence of the first $L$ V-coefficients $\alpha^{(n)}_{[L]}$, which, again due to the contraction principle, satisfies the LDP with speed $n$ and good rate function 
\begin{align}
\label{ratecoeffL}
\mathcal I_{\coeff}^{L}(x) = \inf\{ \mathcal{I}_{\coeff} (\alpha): \alpha_{[L]}=x\} .
\end{align}
We are not able to give an explicit expression for the rate function $\mathcal I_{\coeff}^{L}$. However, for polynomial potentials it is possible to study it directly from the density \eqref{krishnaT}, and we can give a sufficient approximation. 
Recalling the decomposition in Proposition \ref{propclue}, we set
 \[G_j(\alpha_{[j, j+d]}) = G(\alpha_{[j, j+d]}) - G(\alpha^\V_{[j, j+d]})\]
 and
\[F_-^0(\alpha_{[d-1]})= F_-(\alpha_{[d-1]}) - F_-(\alpha_{[d-1]}^\V) \ , \ F_+^0(\alpha) = F_+(\alpha) - F_+(\alpha^\V)\,. \]
Furthermore, let
\begin{align} \label{defWrate}
{\mathcal W_L }(\alpha_{[L-1]}):= F_-^0(\alpha_{[d-1]}) +  
 \sum_{j=0}^{L-d-1}G_j(\alpha_{[j, j+d]})- 
 \sum_{j=0}^{L-1}\log \frac{1-|\alpha_j|^2}{1-|\alpha_j^\V|^2}\,.
\end{align}
Replacing the rate $\mathcal I_{\coeff}^{L}$ by the function $\mathcal W_L$ gives then a small error for coefficients close to those of $\mu_\V$. An essential ingredient for our proofs is the following theorem, which establishes the LDP for the truncated random Verblunsky coefficients and provides a crucial uniform proxy for the rate function. 

\begin{theorem}
\label{BSZmod1}
Assume that $\V$ is a symmetric Laurent polynomial as in \eqref{even}. Then,
$(\alpha^{(n)}_{[L]})$ satisfies a LDP with good rate function ${\mathcal I}_{\coeff}^L$.
Furthermore, there exists a constant $C_\V$, such that for any $L > d$
\begin{align}
\label{ratemaj}
\left|\mathcal I_{\coeff}^L(x) - \mathcal W_L(x)
\right| \leq C_\V \sum _{L-d \leq  k\leq L-1} |x_k  - \alpha_k^\V| =: m^\V_L (x_{[L-d, L-1]})\,.
\end{align}
\end{theorem}


\subsection{From LDPs to sum rules and gems}

This section gives a proof of our main results in Section \ref{sec:our_results}, as a consequence of the large deviation results presented above. 


\subsubsection{Proof of Theorem \ref{Unitabstractgem}} 
For $\V$ a symmetric Laurent polynomial, the assumptions of Theorem \ref{LDPsp} are satisfied and the sequence of spectral measures $\mu_n$ satisfies the LDP with speed $n$ and rate function $\mathcal I_{\meas}$. By the arguments of Section \ref{susec:abstractcoef}, the rate function for the V-coefficients of $\mu_n$ is 
\begin{align}\label{rate=rate}
\mathcal{I}_{\coeff}(\alpha)=\mathcal I_{\meas}(\mu) , 
\end{align}
where $\mu=\psi^{-1}(\alpha)$. By the Dawson-G\"artner Theorem \cite[Theorem 4.6.1]{demboz98}, this rate function satisfies
\begin{align} \label{dawsongaertner}
\mathcal I_{\coeff}(\alpha) = \sup_{L\geq 1} \mathcal I_{\coeff}^{L}(\alpha_{[L-1]}) = \lim_{L\to \infty} \mathcal I_{\coeff}^{L}(\alpha_{[L-1]}) ,  
\end{align}
where the last identity holds since $I_{\coeff}^{L}(\alpha_{[L-1]})$ is increasing in $L$ by relation \eqref{ratecoeffL}. We have therefore by Theorem \ref{BSZmod1}
\begin{align} \label{ratecomparison}
\limsup_{L\to \infty} \mathcal W_L (\alpha_{[L-1]}) & \leq \mathcal I_{\coeff}(\alpha) + \limsup_{L\to \infty }  m^\V_L (\alpha_{[L-d, L-1]}) \\
\notag & \leq  \mathcal I_{\coeff}(\alpha) + 2d C_\V 
\end{align} 
with $\mathcal W_L$ as defined in \eqref{defWrate}, and the analogous lower bound
\begin{align} \label{ratecomparison2}
\liminf_{L\to \infty} \mathcal W_L (\alpha_{[L-1]}) & \geq \mathcal I_{\coeff}(\alpha) - \limsup_{L\to \infty }  m^\V_L (\alpha_{[L-d, L-1]}) \\
\notag & \geq  \mathcal I_{\coeff}(\alpha) - 2d C_\V  . 
\end{align} 
This proves that $\sup_L \mathcal W_L (\alpha_{[L-1]})$ is finite if and only if $\mathcal{I}_{\coeff}(\alpha)=\mathcal I_{\meas}(\mu)$ is finite. By Proposition \ref{propclue},  
\begin{align} \label{RWcomparison}
 {\mathcal R_L }(\alpha_{[L-1]})  - {\mathcal W_L }(\alpha_{[L-1]}) =  F_+(\alpha_{[n-d, n-1]}) - F_+(\alpha^\V_{[n-d, n-1]}).
\end{align}
Since $F_+$ is polynomial, the right hand side may be bounded by a  constant $C$ depending on $\V$, for all $L$ and all sequences $\alpha$. Therefore, $\sup_L \mathcal R_L (\alpha_{[L-1]})$ is finite if and only if $\mathcal I_{\meas}(\mu)$ is finite. 
The conditions (1), (2) and (3) in Theorem \ref{Unitabstractgem} are then equivalent to the finiteness of $\mathcal I_{\meas}(\mu)$. Indeed, note that the Kullback-Leibler divergence $\mathcal K(\mu_\V|\mu)$ is finite if and only if $\mu_\V$ has a density $h$ with respect to $\mu$ and 
\begin{align} \label{KLfinite}
\int \log h(z) \ddd\mu_\V(z) <\infty , 
\end{align}
and in this case $\ddd\mu(z) = h^{-1}(z) \ddd\mu_\V(z)+ \ddd\mu_s(z)$ is the decomposition of $\mu$ with respect to $\mu_\V$. \qed

\subsubsection{Proof of Theorem \ref{Unitabstractnewsumrule}} 
We argue similarly as in the proof of Theorem \ref{Unitabstractgem}, but this time we have additionally as a consequence of $\mu\in M_\V$, that 
\begin{align}
\lim_{L\to \infty } m^\V_L (\alpha_{[L-d, L-1]}) = 0 .  
\end{align} 
This implies by \eqref{rate=rate} and the inequalities \eqref{ratecomparison}, \eqref{ratecomparison2} 
\begin{align*}
\mathcal I_{\meas}(\mu) = \lim_{L\to \infty} {\mathcal W_L}(\alpha_{[L-1]}) . 
\end{align*}
Using again that $\mu\in M_\V$, the right hand side of \eqref{RWcomparison} tends to zero as $L\to \infty$, and we obtain the claimed sum rule 
\begin{align*}
\mathcal I_{\meas}(\mu) = \lim_{L\to \infty} {\mathcal R_L}(\alpha_{[L-1]}) .
\end{align*}
\qed

\subsubsection{Proof of Corollary \ref{Unitnewsumrulesymmetric}} 
Let $\mu\in \mathcal S_1(I)$ be a symmetric probability measure with real coefficients $\alpha$. If $\mathcal I_{\meas}(\mu)=\infty$, inequality \eqref{ratecomparison2} and the uniform bound for \eqref{RWcomparison} imply $\lim_{L\to \infty} \mathcal R_L(\alpha_{[L-1]})=\infty$ and we have identity for $\mu$. Suppose $\mathcal I_{\meas}(\mu)<\infty$. In view of Theorem \ref{Unitabstractnewsumrule}, we need to show that either $\mu\in M_\V$, that is, $|\alpha_k-\alpha_k^\V|\to 0$, or $\mu_{-1}\in M_\V$, such that $|(-\alpha_k)-\alpha_k^\V|\to 0$. 
We apply the following remarkable result, an extension of Rakhmanov's theorem.

\begin{theorem}
\label{Rakhm}
Let $\mu\in \mathcal M_1(\mathbb T)$ with infinite support and decomposition $\ddd \mu(x) = w(x)\ddd x +\ddd\mu_s(x)$ with respect to the uniform measure. 
\begin{enumerate}
\item \cite[Corollary 9.1.11]{simonopuc2}
If $w(z) > 0$ for a.e. $z\in \T$, then
\begin{align}
\lim_{n\to \infty} |\alpha_n| = 0\,.
\end{align}
\item
\cite[Theorem 13.4.4]{simonopuc2}
\label{Rakhm-arc}
Let $\mathbf a \in (0, 1)$,  $\theta_{\mathbf a}= 2 \arcsin \mathbf a$, $\eta \in \mathbb T$ and 
\[\hat a = \Gamma_{a, \eta} = \{z \in \mathbb T : |\arg(\eta z)| > \theta_{\mathbf a}\}\,.\]
If $\hbox{ess}\ \hbox{supp}(\mu) = \Gamma_{a, \eta}$ and $w(z)>0$ for a.e. $z\in \Gamma_{a, \eta}$, then
\begin{align}
\label{nn+1} 
\lim_{n\to \infty} |\alpha_n| = \mathbf a \ , \ \lim_{n\to \infty} (\alpha_n\bar\alpha_{n+1}) = \mathbf a^2\eta\,,
\end{align}
hence for $j \geq 2$ 
\begin{align}
\label{5.26}
\lim_{n\to \infty} \alpha_n \bar\alpha_{n+j}
 = \mathbf a^2 \eta^j\,.
\end{align}
\end{enumerate} 
\end{theorem}

Under assumption (A2), the support of $\mu_\V$ is a single arc $\Gamma_{a, \eta}$ as in case (2) of Theorem \ref{Rakhm}. The measure $\mu_\V$ is symmetric with real coefficients and we may assume without loss of generality $\eta=1$, that is, the point 1 is not contained in the arc. It was shown in Theorem 11.2.4 and Theorem 16.1.5 in \cite{pastur_shcherbina} that $\mu_\V$ has a density $v$ with respect to the uniform law on $\T$, which satisfies $v>0$ a.e. on $\Gamma_{a, \eta}$.   Consequently, Theorem \ref{Rakhm} gives  
\begin{align}\label{rakhmanovlimit1}
\lim_{n\to \infty} |\alpha^\V_n| = \mathbf a \ , \ \lim_{n\to \infty} (\alpha^\V_n\alpha^\V_{n+1}) = \mathbf a^2 , 
\end{align}
which implies that the limit $\gamma^\V:=\lim_{n\to \infty}\alpha^\V_n$ exists.  
Now, if $\mu$ is a symmetric measure with real V-coefficients $\alpha$ and with $\mathcal I_{\meas}(\mu)<\infty$, then the essential support of $\mu$ has to coincide with the essential support of $\mu_\V$. Additionally, since $\mathcal K(\mu_\V|\mu)$ is finite, we have as in \eqref{KLfinite} that $\ddd\mu(z) = h^{-1}(z) \ddd\mu_\V(z)+ \ddd\mu_s(z)$ with $h(z)\in (0,\infty)$ a.e. on $\Gamma_{a, \eta}$. Then{\fab ,} the Lebesgue decomposition is $\ddd\mu(z) =  w(z)\ddd z+ \ddd\mu_s(z)$ with $w(z)=h^{-1}(z)v(z)$ positive a.e. on $\Gamma_{a, \eta}$. Another application of Theorem \ref{Rakhm} yields that the limit  $\gamma:=\lim_{n\to \infty}\alpha_n$ exists as well and $\gamma^2 = \mathbf a = (\gamma^\V)^2$. That is, we have either $\gamma=\gamma^\V$ or $\gamma=-\gamma^\V$. If $\gamma=\gamma^\V$, the difference of coefficients satisfies $|\alpha_k-\alpha_k^\V|\to 0$. 
If $\gamma=-\gamma^\V$, we get that $|\alpha_k+\alpha_k^\V|\to 0$. 
\qed

\section{Proofs}
\label{sec:proofs}
\subsection{Spectral LDP}
\label{sec:LDPproofspectral}

The LDP in Theorem \ref{LDPsp} follows with similar arguments as the spectral LDP in \cite{gamboanag2021}. In fact, it is possible to reduce the LDP on the unit circle to a suitable LDP on the real line. For this, we assume that $1$ is not contained in the support $I$ of the limit measure $\mu_\V$, otherwise we may just rotate the circle. Let $r$ be the mapping $\theta\mapsto e^{i \theta}$ from $[0,2\pi)$ to the unit circle, which induces a bijection from $\mathcal M_1([0,2\pi))$ to $\mathcal M_1(\T)$, which we denote again by $r$. Note that $r$ is continuous with respect to the weak topology, but $r^{-1}$ is not. For a measure $\mu\in \mathcal S_1(I)$, the pushforward $r^{-1}(\mu)$ yields a measure in the set $\overline{\mathcal S}_{1}(\overline{I})$, which contains all measures on $[0,2\pi)$ with support $\overline{J}\cup \overline{E}$, with $\overline{J}\subset  \overline I = r(I)$ and $\overline{E}$ an at most coutable subset of $[0,2\pi)\setminus \overline I$. 

From the distribution of the random spectral measure $\mu_n$ in Section \ref{sec:randomization}, we directly obtain the distribution of 
\begin{align}
\bar\mu_n = r^{-1}(\mu_n) = \sum_{i=1}^n w_i\delta_{\theta_i}
\end{align}
as follows: the weights $(w_1,\dots ,w_n)$ are independent of the support points and uniformly distributed on the unit simplex, the support points $\theta_1,\dots ,\theta_n$ have the joint Lebesgue density 
\begin{align*}
Z_n^{-1} \Delta (\theta_1,\dots ,\theta_n)^2 \exp\left(-n\sum_{i=1}^n V(\theta_i) \right) .
\end{align*} 
The potential $V:\mathbb{R}\to (-\infty,\infty]$ is given by 
\begin{align}
V(\theta) = \begin{cases} \V(e^{i\theta})  & \text{ if }\theta\in [0,2\pi),\\
+\infty & \text{ otherwise. } \end{cases}
\end{align}
The equilibrium measure is $\bar\mu_V=r^{-1}(\mu_\V)$.  
For a potential $\V$ satisfying assumptions (A1), (A2') and (A3), the potential $V$ satisfies the analogous assumptions (A1), (A2) and (A3) in \cite{gamboanag2021}. Theorem 4.3 therein yields the following LDP.

\begin{theorem}
\label{LDPsp2}
If $\V$ satisfies assumptions (A1), (A2') and (A3), then the sequence of measures $\bar\mu_n$ satisfies the LDP in $\mathcal M_1(\mathbb{R})$ with speed $n$ and good rate function
\begin{align}
\label{35-2}
\overline{\mathcal I}_{\meas}(\bar\mu) = \mathcal K(\bar\mu_V| \bar\mu) + \sum_{\theta\in E(\bar\mu)} \mathcal F_\V (e^{i\theta})\,,
\end{align}
if $\bar\mu \in \mathcal S_1(\bar I)$, and $\overline{\mathcal I}_{\meas}(\bar\mu)=\infty$ otherwise. 
\end{theorem}

Note that we have $\mathcal K(\bar\mu_\V| \bar\mu)= \mathcal K(\mu_\V| r(\bar\mu))$ by the reversible entropy principle. 
The random measures $\bar\mu_n$ are elements of $\mathcal M_1([0,2\pi))$ and the rate function $\overline{\mathcal I}_{\meas}$ is infinite on the complement of this set. We may therefore restrict the LDP in Theorem \ref{LDPsp2} to the space $\mathcal M_1([0,2\pi))$.  
Since $r$ is a continuous mapping from $\mathcal M_1([0,2\pi))$ to $\mathcal M_1(\T)$, we obtain the LDP in Theorem \ref{LDPsp} by the contraction principle.

\subsection{Coefficient LDP}

Recall that as a consequence of the LDP in Theorem \ref{LDPsp} and the contraction principle, the sequence $(\alpha^{(n)}_{[L-1]})_{n\geq 1}$ of the first $L$ V-coefficients of $\mu_n$, satisfies the LDP in $\pi_L(\mathcal R)$ with speed $n$ and good rate function  
\[\mathcal I_{\coeff}^L(x) = \inf \{\mathcal I_\meas(\mu) ; (\psi(\mu))_{[L-1]} = x\}\,.\]
The set $\D^L$ is a strict subset of the ambient space $\pi_L(\mathcal R)$. If $x\in \pi_L(\mathcal R)\setminus \D^L$ and $(\psi(\mu))_{[L-1]} =x$, the measure $\mu$ is finitely supported. In this case, $\mathcal I_\meas(\mu)=\infty$. Therefore, the rate function $\mathcal I_{\coeff}^L$ is infinite outside of $\D^L$. The first $L$ V-coefficients of $\mu_n$ are elements of $\D^L$ with probability 1 as soon as $n> L$, since then $\mu_n$ has at least $L+1$ support points almost surely. By Lemma  4.1.5 in \cite{demboz98}, we obtain that the sequence $(\alpha^{(n)}_{[L-1]})_{n> L}$ satisfies the LDP in $\D^L$ with speed $n$ and good rate function the restriction of $\mathcal I_{\coeff}^L$ to $\D^L$. 

For $x\in \D^L$, we denote by $B_{\delta, L}(x)$ the ball around $x$ in $\D^L$ with radius $\delta $ in the sup-norm. As a consequence of the LDP, we have for any $x\in \D^L$ 
\begin{align} \label{weakballlimit}
\mathcal I_{\coeff}^L(x) & = \lim_{\delta \to 0}\limsup_{n\to \infty} -\frac{1}{n}\log \PP\sn_\V\big(\alpha\sn_{[L]}\in B_{\delta, L}(x)\big) \\
& = \lim_{\delta \to 0}\liminf_{n\to \infty} -\frac{1}{n}\log \PP\sn_\V\big(\alpha\sn_{[L]}\in B_{\delta, L}(x)\big)  . \notag 
\end{align}
This allows to study the rate function by estimating the probability of small balls from the expression of the density in \eqref{krishnaT}. Without loss of generality, we may assume that the balls $B_{\delta, L}$ have a positive distance to the boundary of $\D^L$. 

Up to a change of the normalizing constant, the density of $\alpha^{(n)}= \alpha_{[n-1]}^{(n)}$ is given by 
\begin{align*}
(\mathcal Z_n^\V)^{-1}\exp \left( - n H_n (\alpha_{[n-1]}) \right)  
\end{align*}
where
\begin{align}
H_n (\alpha_{[n-1]}) =  \tr \V(\mathcal G_n(\alpha)) - \tr \V(\mathcal G_n^\V)- \sum_{k=0}^{n-2} \left(1 - \frac{k+2}{n}\right) \log \frac{1- |\alpha_k|^2}{1- |\alpha_k^\V|^2}\,. 
\end{align}
Now, owing to Proposition \ref{propclue} and with the centering as in Section \ref{sec:our_results}, we have
\begin{align}
\notag
H_n (\alpha_{[n-1]}) & =  F_-^0(\alpha_{[d-1]}) + F_+^0 (\alpha_{[n-d, n-1]}) + \sum_{k=0}^{n-1-d}G_k(\alpha_{[k, k+d]})\\
\label{likelihood}
&\qquad - \sum_{k=0}^{n-2} \left(1 - \frac{k+2}{n}\right) \log \frac{1- |\alpha_k|^2}{1- |\alpha_k^\V|^2} , 
\end{align}
where we recall that as in \eqref{defGK}, 
\begin{align*}
G_k(\alpha_{[k, k+d]}) = \sum_{j=1}^d \Gamma_{j}(\alpha_{[k,k+d]})- \Gamma_{j}(\alpha^\V_{[k,k+d]}) . 
\end{align*}
The proof now proceeds similarly to the one of \cite[Theorem 3.3]{breuersizei2018} and \cite[Theorem 3.6]{breuersizei2018}.   
For $n>L+d>2d$, we decompose the log-likelihood into three parts, one depending only on the first $L$ coefficients, for which the function $\mathcal W_L$ is the main contribution, one part $H_n^{\mathrm{h}}$ depending only on higher order terms and a mixed term $H_n^{\mathrm{m}}$: 
\begin{align} \label{likelihooddecomposition} 
H_n (\alpha_{[n-1]}) =  \mathcal W_L (\alpha_{[L-1]}) + \ell_0(\alpha_{[L-1]}) + H^{\mathrm{m}}_n (\alpha_{[L-d,L+d]})+H^{\mathrm{h}}_n (\alpha_{[L,n-1]})
\end{align}
where
\begin{align*}
\ell_0(\alpha_{[L-1]}) & = \sum_{k=0}^{L-2} \frac{k+2}{n} \log \frac{1- |\alpha_k|^2}{1- |\alpha_k^\V|^2} , \\ 
H^{\mathrm{m}}_n (\alpha_{[L-d,L+d]}) & = \sum_{k=L-d}^{L} G_k(\alpha_{[k,k+d]}),  \\
H^{\mathrm{h}}_n (\alpha_{[L,n-1]}) & = F_+^0 (\alpha_{[n-d, n-1]}) + \sum_{k=L+1}^{n-1-d} G_k(\alpha_{[k,k+d]}) 
- \sum_{k=L+1}^{n-2} \left(1 - \frac{k+2}{n}\right) \log \frac{1- |\alpha_k|^2}{1- |\alpha_k^\V|^2}.  \\
\end{align*}

We can give suitable bounds for these terms, after some more modifications. For a vector $\alpha$ of V-coefficients, denote by $\alpha^{\V,L}$ the vector with the first $L$ coefficients $\alpha_k$ replaced by $\alpha_k^\V$. Set 
\begin{align*}
\tilde H^{\mathrm{m}}_n (\alpha_{[L-d,L+d]}) & = H^{\mathrm{m}}_n (\alpha_{[L-d,L+d]}) - H^{\mathrm{m}}_n (\alpha^{\V,L}_{[L-d,L+d]}) , \\
\tilde H^{\mathrm{h}}_n (\alpha_{[L,n-1]}) & = H^{\mathrm{h}}_n (\alpha_{[L,n-1]}) + H^{\mathrm{m}}_n (\alpha^{\V,L}_{[L-d,L+d]}) . 
\end{align*} 
Then $\tilde H^{\mathrm{h}}_n$ is still independent of $\alpha_{[L-1]}$ and the decomposition in \eqref{likelihooddecomposition} still holds, with $H^{\mathrm{m}}_n, H^{\mathrm{h}}_n$ replaced by $\tilde H^{\mathrm{m}}_n,\tilde H^{\mathrm{h}}_n$. 

There exists a constant $C$ depending on $x,\delta,L$ and $\V$ such that 
\begin{align}\label{mixedtermbound0}
\ell_0(\alpha_{[L-1]}) \leq \frac{C}{n} 
\end{align}
for all $\alpha_{[L-1]}\in B_{\delta, L}(x)$, since this ball is bounded away from the boundary of $\D^L$. For $\tilde H^{\mathrm{m}}_n$ we have 
\begin{align} \label{mixedtermbound}
& | \tilde H^{\mathrm{m}}_n (\alpha_{[L-d,L+d]}) | \notag  \\
&  \leq \sum_{k=L-d}^L \sum_{j=1}^d\left|  \Gamma_{j}(\alpha_{[k,k+d]}) - \Gamma_{j}(\alpha^{\V,L}_{[k,k+d]}) \right| \notag \\
& \leq \sum_{k=L-d}^L \sum_{j=1}^d \sum _{k \leq k_1\leq k_2\leq \dots \leq k_{2j}\leq k+d} |c_{k_1, \dots, k_{2j}}| |\alpha_{k_1}\bar \alpha_{k_2}\dots \alpha_{k_{2j-1}}\bar \alpha_{k_{2j}} - \alpha^{\V,L}_{k_1}\bar \alpha^{\V,L}_{k_2}\dots \alpha^{\V,L}_{k_{2j-1}}\bar \alpha^{\V,L}_{k_{2j}}| . 
\end{align}
In the remaining difference, any coefficient with index at least $L$ appears in both products, may be factored out and bounded by 1. An extensive use of the triangle inequality and of the boundedness of V-coefficients yields the bound
\begin{align} \label{mixedtermbound2}
 | \tilde H^{\mathrm{m}}_n (\alpha_{[L-d,L+d]}) | \leq C_\V \sum _{L-d \leq k\leq L-1} |\alpha_k  - \alpha_k^\V |= m^\V_L (\alpha_{[L-d, L-1]})\,.
\end{align}
We will consider ratios of probabilities, so that we can ignore the normalizing constant and consider the measure $\widetilde{\PP}\sn_\V =(\tilde Z^V) \PP \sn_\V$. 
Using the above decomposition of $H_n$, we have
\begin{align*}
& \widetilde{\PP}\sn_\V \big(\alpha\sn_{[L-1]}\in B_{\delta, L}(x) \big ) \\
& = \int_{\tilde B(x)} \exp\left\{-n \left(  \mathcal W_L (\alpha_{[L-1]}) + \ell_0(\alpha_{[L-1]}) + \tilde H^{\mathrm{m}}_n (\alpha_{[L-d,L+d]})+\tilde H^{\mathrm{h}}_n (\alpha_{[L,n-1]})\right) \right\} d\lambda_n ,
\end{align*}
where the integral is over $\tilde B(x) = B_{\delta,L} (x)\times \D^{n-1-L}\times \T$. Here, we wrote $\lambda_n$ for the Lebesgue measure on $\D^{n-1}\times T$. After bounding $\ell_0$ as in \eqref{mixedtermbound0} and $\tilde H^{\mathrm{m}}_n $ as in \eqref{mixedtermbound2}, the exponential is a sum of terms depending on $\alpha_{[L-1]}$ and terms depending on $\alpha_{[L,n-1]}$, which implies that this upper bound factorizes:
\begin{align*}
\widetilde{\PP}\sn_\V \big(\alpha\sn_{[L-1]}\in B_{\delta, L}(x) \big) & \leq e^C \int_{B_{\delta,L}(x)} \exp\left\{ -n \left( \mathcal W_L (\alpha_{[L-1]}) + m^\V_L (\alpha_{[L-d, L-1]})\right)\right\} \ddd \lambda_L \\
& \quad \times  
\int \exp\left\{-n \tilde H^{\mathrm{h}}_n (\alpha_{[L,n-1]})\right\} d\lambda_{n-L} .
\end{align*} 
Looking at the ratio of probabilities, we have then 
\begin{align*}
\frac{1}{n} \log \frac{{\PP}\sn_\V \big(\alpha\sn_{[L-1]}\in B_{\delta, L}(x) \big)}{{\PP}\sn_\V \big(\alpha\sn_{[L-1]}\in B_{\delta, L}(\alpha^\V_{[L-1]}) \big)} 
&  = \frac{1}{n} \log \frac{\widetilde{\PP}\sn_\V \big(\alpha\sn_{[L-1]}\in B_{\delta, L}(x) \big)}{\widetilde{\PP}\sn_\V \big(\alpha\sn_{[L-1]}\in B_{\delta, L}(\alpha^\V_{[L-1]}) \big)} 
 \\
& \leq \sup_{ y\in  B_{\delta,L}(x)}  \left(  - \mathcal{W}_L(y) + m^\V_L (y_{[L-d, L-1]})\right) \\
& \quad - \inf_{  y\in  B_{\delta,L}(\alpha^\V_{[L-1]})}  \left( -  \mathcal{W}_L(y) - m^\V_L (y_{[L-d, L-1]})\right)\\
&\quad + \frac{2C}{n} + \frac{1}{n}\log \frac{\lambda_L(B_{\delta, L}(x))}{\lambda_n(B_{\delta, L}(\alpha^\V_{[L-1]}))} . 
\end{align*}
The functions $\mathcal{W}_L$ and $m^\V_L$ are uniformly continuous on any compact subset of $\D^L$, which implies
$$
\lim_{\delta \to 0} \sup_{ y\in  B_{\delta,L}(x)}  \left(  - \mathcal{W}_L(y) + m^\V_L (y_{[L-d, L-1]})\right)  
= - \mathcal{W}_L(x) + m^\V_L (x_{[L-d, L-1]}),
$$
and
$$
\lim_{\delta \to 0} \inf_{ y\in  B_{\delta,L}(\alpha^\V_{[L-1]})}  \left(  - \mathcal{W}_L(y) + m^\V_L (y_{[L-d, L-1]})\right)  \\
= - \mathcal{W}_L(\alpha^\V_{[L-1]})) + m^\V_L (\alpha^\V_{[L-1]})) = 0 .
$$
For the ratio of probabilities this implies
\begin{align*} 
\lim_{\delta\to 0} \limsup_{n\to \infty} \frac{1}{n}  \log \frac{{\PP}\sn_\V \big(\alpha\sn_{[L-1]}\in B_{\delta, L}(x) \big)}{{\PP}\sn_\V \big(\alpha\sn_{[L-1]}\in B_{\delta, L}(\alpha^\V_{[L-1]}) \big)} 
& \leq  - \mathcal{W}_L(x) + m^\V_L (x_{[L-d, L-1]}) .
\end{align*}
Since ${\PP}\sn_\V (\alpha\sn_{[L-1]}\in B_{\delta, L}(\alpha^\V_{[L-1]}) )$ converges to 1, this implies by \eqref{weakballlimit} the lower bound 
\begin{align*}
\mathcal I_{\coeff}^L(x) & = \lim_{\delta \to 0}\liminf_{n\to \infty} -\frac{1}{n}\log \PP\sn_\V\big( \alpha\sn_{[L]}\in B_{\delta, L}(x)\big)
\geq \mathcal{W}_L(x) - m^\V_L (x_{[L-d, L-1]}) .
\end{align*}
The upper bound follows by analogous arguments. 
\hfill $\Box$ 
\newpage
\printnomenclature

\newpage
\input{appendix_september_25.tex}

{\fab \subsection*{Acknowledgement}
Support from the ANR-3IA Artificial and Natural Intelligence
Toulouse Institute is gratefully acknowledged.}

\bibliographystyle{alpha}

\bibliography{BIBLIOGRAPHIE/books,BIBLIOGRAPHIE/articles,BIBLIOGRAPHIE/inbooks,BIBLIOGRAPHIE/thesis,BIBLIOGRAPHIE/proceedings,BIBLIOGRAPHIE/rapports}
\end{document}

%% file: appendix_september_25.tex
{\fab
\appendix
\section{Sum rules for particular potentials}
\label{sec:append}
We next recall several sum rules established in \cite{breuersizei2018} and explicitly computed in \cite{du2023}, but we present them here in a different form, relying on a large deviation principle (LDP) involving entropies.
\subsection{The $(1,0)$ case.} The potential is $\V(z) = \mathrm{Re}(z)$. Since
\begin{equation}
 \label{eq:potentialcomputation10}
\tr\,\V\left(\mathcal G_L(\alpha)\right)=\tr\,\mathcal G_L(\alpha) = \sum_{k= -1}^{L-2} -\alpha_k \bar\alpha_{k+1}= \bar\alpha_0  
- \sum_{k=0}^{L-2} \alpha_k \bar\alpha_{k+1}
\end{equation} 
we have 
\begin{align}
F_-^{1,0} (\alpha_0) =\mathrm{Re}( \alpha_0)\ , \ 
G^{1,0}(\alpha_{[j,j+1]}) = - \mathrm{Re} (\alpha_j\bar\alpha_{j+1})\,.
\end{align}

Recall the following gem and sum rule  both given in
\cite[Theorem 2.8.1]{simonopuc1} and \cite[Theorem 4.5]{du2023}. Here,  the left hand side is reformulated with the help of \eqref{Hvalue}. We have,
 \begin{align}
 \label{gem10}
 \mathcal K(\mu^{1,0}| \mu) < \infty\Longleftrightarrow \sum_{k=1}^\infty |\alpha_{k+1} - \alpha_{k}|^2 + |\alpha_k|^4 < \infty\, ,
 \end{align}
 and more precisely
\begin{mdframed}
 \begin{align}
 \label{sr10}
 \mathcal K(\mu^{1,0}| \mu) = \frac{1}{2} - \log 2
 + \frac{1}{2}\sum_{k=0}^\infty |\alpha_k - \alpha_{k-1}|^2  +\sum_{k=0}^\infty\left( -\log(1-|\alpha_k|^2) - |\alpha_k|^2\right) .
\end{align}
\end{mdframed} 

\subsection{The $(1,1)$ case.} 
The potential is $\V(z) = \tfrac{1}{2}\mathrm{Re}(z^2)$. 
To begin with note that,
\begin{eqnarray}
\tr\,\mathcal G_L(\alpha)^2 &=&\sum_{k=0}^{L-2}  -2 \alpha_{k-1}\bar\alpha_{k+1} +2 \alpha_{k-1}|\alpha_k|^2\bar\alpha_{k+1} + \alpha_{k-1}^2 (\bar\alpha_k)^2 \nonumber\\
&=& 2\bar\alpha_0 -2 |\alpha_0|^2\bar\alpha_1+ (\bar\alpha_0)^2 + \sum_{k=0}^{L-3}
-2\alpha_k\bar\alpha_{k+2} +2\alpha_k|\alpha_{k+1}|^2\bar\alpha_{k+2} +\alpha_k^2(\bar\alpha_{k+1})^2 , 
\label{eq:potentialcomputation11}
\end{eqnarray}
so that 
\begin{align*}
F_-^{1,1}(\alpha_{[1]}) &=\mathrm{Re} (\alpha_0)-|\alpha_0|^2\mathrm{Re}(\alpha_1)+ \frac{1}{2} \mathrm{Re}(\alpha_0)^2 \\ 
G^{1,1}(\alpha_{[j,j+2]}) &= 
-\mathrm{Re}(\alpha_j\bar\alpha_{j+2})(1 - |\alpha_{j+1}|^2) + \frac{1}{2}\mathrm{Re}\left((\alpha_j\bar\alpha_{j+1})^2\right)\,.
\end{align*}
By \eqref{Hvalue}, we have,
\begin{align}
\mathcal K(\mu^{1,1}|\mu) & =\mathcal K(\mu^{1,1}|\UNIF) - 2\int_0^{2\pi} (1-\cos^2 \theta) \log w(\theta) \frac{\ddd \theta}{2\pi} \notag \\ 
& = 1 - \log 2 - 2\int_0^{2\pi} (1-\cos^2 \theta) \log w(\theta) \frac{\ddd \theta}{2\pi} ,
\end{align}
leading to the following gem and sum rule (see \cite[Theorem 4.7]{du2023}), 
 \begin{align}
 \label{gem11}
 \mathcal K(\mu^{1,1}| \mu) < \infty \Longleftrightarrow \sum_{k=1}^\infty |\alpha_{k+2} - \alpha_{k}|^2 + |\alpha_k|^4 < \infty\,,
 \end{align}
 and more precisely
\begin{mdframed}
\begin{align}
\notag
 \mathcal K(\mu^{1,1}| \mu) &= \frac{1}{4}
 - \log 2  +\sum_{k=0}^\infty\left( -\log(1-|\alpha_k|^2) - |\alpha_k|^2 -\frac{1}{2}|\alpha_k|^4\right)\\
 \label{sr11}
 &\quad + \sum_{k=0}^\infty |\alpha_k\alpha_{k-1}|^2 + \frac{1}{2}\sum_{k=0}^\infty (1-|\alpha_k|^2) |\alpha_{k+1}-\alpha_{k-1}|^2
\\
\notag
&\quad + \frac{1}{8} \sum_{k=0}^\infty \left(\left(2|\alpha_k|^2 - |\alpha_k - \alpha_{k-1}|^2\right)^2 + \left(2|\alpha_{k-1}|^2 - |\alpha_k - \alpha_{k-1}|^2\right)^2\right).
\end{align}
\end{mdframed}

\subsection{The $(2,0)$ case.} Here, we use that 
 \begin{equation}
 \label{eq:gpotential20}
 \V^{2,0} = -\frac{1}{3}\V^{1,1} + \frac{4}{3} \V^{1,0}\,,
 \end{equation}
so we  easily get
\begin{align}
\notag
F_-^{2,0}(\alpha_0, \alpha_1) &= -\frac{1}{3}F_-^{1,1}(\alpha_0, \alpha_1) + \frac{4}{3}F^{1,0}_- (\alpha_0)\\
G^{2,0}(\alpha_{[j, j+2]}) &=  -\frac{1}{3}G^{1,1}(\alpha_{[j,j+2]})
+ \frac{4}{3}G^{1,0} (\alpha_{[j,j+1]} )\,.
\end{align}
So that, using \eqref{g=0}, we get
\begin{align}
\notag \mathcal K(\mu^{2,0}|\mu) & = \mathcal K(\mu^{2,0}|\UNIF) - \frac{2}{3}\int_0^{2\pi} (1 - \cos \theta)^2 \log w(\theta) \frac{\ddd \theta}{2\pi} \\
& = \frac{7}{3} -\log 6 - \frac{2}{3}\int_0^{2\pi} (1 - \cos \theta)^2 \log w(\theta) \frac{\ddd \theta}{2\pi} \,,
\end{align}
implying the following gem and sum rule (borrowed from \cite[Theorem 4.8]{du2023}),

\begin{align}
\label{gem20}
\mathcal K(\mu^{2,0}|\mu) <\infty \ \Longleftrightarrow \ \sum_{k=0}^\infty\left(|\alpha_{k+2}-2\alpha_{k+1} +\alpha_k|^2 + |\alpha_k|^6\right) <\infty\, ,
\end{align}
\begin{mdframed}
\begin{align}
 \notag
 \mathcal K(\mu^{2,0}|\mu) &=\frac{19}{12} - \log 6 +\sum_{k=0}^\infty\left( -\log(1-|\alpha_k|^2) - |\alpha_k|^2 -\frac{1}{2}|\alpha_k|^4\right)\\
 \notag
 &\quad +\frac{1}{6}\sum_{k=0}^\infty \left(|\alpha_k|^2 - |\alpha_{k-1}|^2\right)^2\\
 \notag
 &\quad +\frac{1}{6}\sum_{k=0}^\infty (1- |\alpha_k|^2) |\alpha_{k+1}-2\alpha_k + \alpha_{k-1}|^2\\
 \label{sr20}
 &\quad+ \frac{1}{12}\sum_{k=0}^\infty\left(6 |\alpha_k|^2 + 
 6 |\alpha_{k-1}|^2-|\alpha_k -\alpha_{k-1}|^2\right)|\alpha_k -\alpha_{k-1}|^2\,.
\end{align}
\end{mdframed}
\section{Computation and proofs in the ungapped case}
\label{append:ungaped}
\subsection{The (1,0) case: Gross-Witten ensemble}
\label{sec:GWi}
%
\label{potGW}
Note that in comparison with \cite{gamboanag2017,gamboanag2023}, here the sign of $\g$ is different. Notice further that this is consistent with the ungapped case treated in \cite{simonopuc1}. Often, only the case $\g\geq 0$ is considered, changing the parameter sign is equivalent to the change  $U\mapsto -U$. 
If $\g =0$ we recover the CUE and the V-coefficients are independent. 
If $\g \neq 0$, we loose independence.  More precisely, the joint distribution of the V-coefficients is given by (\ref{krishnaT}) with 
\begin{align*}
\tr\, \V_\g^{1,0}(\mathcal G_n) =  \g \mathrm{Re} \left(\alpha_0 - \sum_{k=0}^{n-1} \alpha_k \bar \alpha_{k-1}\right)\,.
\end{align*}
For $|\g| \leq 1$ we are in the ungapped or strongly coupled phase. The equilibrium measure is the Gross-Witten measure $\GW_\g$, we also use the notation $\mu_\g^{1,0}=\GW_\g$. It is supported by $\mathbb T$, with density 
\begin{align}
\label{GW-}
\rho_\g^{1,0}( e^{i\theta}) 
&= \frac{1}{2\pi} (1 - \g \cos \theta)\,.
\end{align}
The only nontrivial moments are of order  $\pm 1$ and the V-coefficients are
\begin{equation}
\label{alphalimGW}
\alpha_n^\g = \begin{cases}\displaystyle  -\frac{x_+ - x_-}{x_+^{n+2} - x_-^{n}} & \hbox{if} \ |\g| < 1\\
\displaystyle \frac{-1}{n+2}& \hbox{if} \ |\g| = 1\,,
\end{cases}
\end{equation}
where $x_\pm = \g^{-1} \pm \sqrt{\g^{-2}-1}$ (see Simon \cite[p. 86]{simonopuc1}).

The first sum rule relative to the Gross-Witten measure $\GW_\g$ was discovered by Simon for $\g = 1$ \cite[Theorem 2.8.1]{simonopuc1}. The extension to $ |\g| \leq 1$ can also be found in \cite[Corollary 5.4]{gamboanag2017}.
A computation of $H(\g)$ is in \eqref{Hvalue}. 


\subsection{The (1,1) case}

Notice that  the potential is invariant by $\theta \mapsto 2\pi -\theta$, by $\theta \mapsto \pi -\theta$ and also by the simultaneous change $\g \mapsto - \g\  ,\ \theta \mapsto \frac{\pi}{2} + \theta$. It is then enough to consider $\g\geq  0$.
Let give the proof on Theorem \ref{th:ungap} in this case.
We make use of the fact that for $\g \in [0,1]$
\begin{align}\label{linearcomb}
\mu_\g^{1,1} = (1-\g)\UNIF + \g \mu_{1}^{1,1} 
\end{align}
and combine the sum rules relative to the uniform measure and to $\mu_{1}^{1,1}$. It follows from \cite[Proposition 8.2]{gamboanag2023} that 
\begin{align*}
\mathcal K(\mu_\g^{1,1} |\mu) & = (1- \g )\mathcal K(\UNIF |\mu) +  \g  \mathcal K (\mu_{1}^{1,1}|\mu) \\
& \quad - (1- \g ) \mathcal K(\UNIF |\mu_\g^{1,1}) -  \g  \mathcal K(\mu_{1}^{1,1}|\mu_\g^{1,1})\notag  .
\end{align*}
The last two divergences are finite, so that $\mathcal K(\mu_\g^{1,1} |\mu)$ is finite if and only if both $\mathcal K(\UNIF |\mu)<\infty $ and $\mathcal K (\mu_{1}^{1,1}|\mu)<\infty$. 
The first condition is equivalent to $\alpha \in \ell^2$ by \eqref{SVsum}, the second one is weaker by (\ref{gem10}). 
Combining the Szeg\H{o}-Verblunsky sum rule \eqref{SVsum} and (\ref{sr11}), we get  
\begin{align}
\notag
\mathcal K(\mu_\g^{1,1}| \mu)&= (1-\g)\sum_{k=0}^\infty -\log(1-|\alpha_k|^2) + \g\left(\frac{1}{4} - \log 2\right)\\
\notag
&+ \g \sum_{k=0}^\infty\left( -\log(1-|\alpha_k|^2) - |\alpha_k|^2 -\frac{1}{2}|\alpha_k|^4\right)\\
 \label{sr11g}
 &+ \g \sum_{k=0}^\infty |\alpha_k\alpha_{k-1}|^2 + \frac{\g}{2}\sum_{k=0}^\infty (1-|\alpha_k|^2) |\alpha_{k+1}-\alpha_{k-1}|^2
\\
\notag
&+  \frac{\g}{8} \sum_{k=0}^\infty \left(\left(2|\alpha_k|^2 - |\alpha_k - \alpha_{k-1}|^2\right)^2 + \left(2|\alpha_{k-1}|^2 - |\alpha_k - \alpha_{k-1}|^2\right)^2\right)\\
\notag
&- (1- \g ) \mathcal K(\UNIF\mid\mu_\g^{1,1}) - \g  \mathcal K(\mu_1^{1,1}\mid\mu_\g^{1,1})\,.
\end{align}
It remains to apply  (\ref{522}) and (\ref{522'}). 
\qed
%
%
%
\subsection{The (2,0) case}
Let us begin with an interesting lemma concerning the support of
$\mu_\g^{2,0}$. 
\begin{lemma}
\label{L62}
For $\g \in [-3/5, 1]$ the equilibrium measure  $\mu_\g^{2,0}$ has full support with density 
\begin{align*} 
\rho^{2,0}_\g
 (e^{i\theta}) = \frac{1 }{2\pi}\left(1 - \frac{4\g}{3}\cos\theta + \frac{\g}{3}\cos 2\theta\right)\,.
\end{align*}
\end{lemma}


\begin{proof}
We apply Lemma \ref{lemps}. 
Set 
\[1 - \frac{4\g}{3}\cos\theta + \frac{\g}{3}\cos 2\theta =:  f_\g (\cos \theta)\]
where
\[f_\g (c) = 1 - \frac{\g}{3} - \frac{4\g}{3} c+ \frac{2\g}{3}c^2\]
is an order two polynomial. If $\g \in (0,1)$, then $f_\g$  has no roots and is positive on $[-1,1]$. 
If $\g \notin [0,1]$, $f_\g$ has two roots
 \[c_\pm = 1 \pm \sqrt{\frac{3}{2}\left(1 - \g^{-1}\right)}\,,\]
with $c_- < 1 < c_+$.
 There is a sign change of $f_\g$ in $[-1, 1)$ if and only if $c_- \in [-1,1)$, i.e. $\g<-3/5$. 
This result is consistent with \cite{houart1990}. \qed
\end{proof} 
\medskip

%
\noindent
Let  now give the proof on Theorem \ref{th:ungap} in this case.
%
We again use that the equilibrium measure is a linear combination 
\begin{align}\label{linearcomb1}
\mu_\g^{2,0} = (1-\g)\UNIF + \g \mu^{2,0}_1  
\end{align}
for $\g \in [0,1]$ and combine the sum rules relative to the uniform measure and to $\mu_{1}^{2,0}$. The left hand side is by \cite[Proposition 8.2]{gamboanag2023}  
\begin{align*}
\mathcal K(\mu_\g^{2,0} |\mu) & = (1- \g )\mathcal K(\UNIF |\mu) +  \g  \mathcal K (\mu_{1}^{2,0}|\mu) \\
& \quad - (1- \g ) \mathcal K(\UNIF |\mu_\g^{2,0}) -  \g  \mathcal K(\mu_{1}^{2,0}|\mu_\g^{2,0})\notag  ,
\end{align*}
again with the last two divergences finite. Finiteness of $\mathcal K(\UNIF |\mu)$ is equivalent to $\alpha \in \ell^2$ by \eqref{SVsum}, which then implies finiteness of $K (\mu_{1}^{2,0}|\mu)$ by  (\ref{gem20}).  
We then combine the Szeg\H{o}-Verblunsky sum rule \eqref{SVsum} and the sum rule (\ref{sr20}). The remaining divergences are given by   
\begin{align*}
\mathcal K(\UNIF | \mu_\g^{2,0}) &= - \int_0^{2\pi} \log\left(1 - \frac{4\g}{3} \cos\theta + \frac{\g}{3}\cos 2\theta\right) \frac{\ddd \theta}{2\pi}= - K(\g)\,.
\end{align*}
and
\begin{align*}
\notag
\mathcal K(\mu^{2,0}|\mu_\g^{2,0})&= \int_0^{2\pi} \log \frac{1 - \frac{4}{3} \cos\theta + \frac{1}{3}\cos 2\theta}{1- \frac{4\g}{3} \cos\theta + \frac{\g}{3}\cos 2\theta}(1- \frac{4}{3} \cos\theta + \frac{1}{3}\cos 2\theta)\frac{\ddd \theta}{2\pi}\\
&= K(1) - K(\g) + I(\g)= -\log 6 
-K(\g)+ I(\g)
\end{align*}
with
\begin{align*}
I(\g) &=\int_0^{2\pi} \log \frac{1 - \frac{4}{3} \cos\theta + \frac{1}{3}\cos 2\theta}{1- \frac{4\g}{3} \cos\theta + \frac{\g}{3}\cos 2\theta}\left(- \frac{4}{3} \cos\theta + \frac{1}{3}\cos 2\theta\right)\frac{\ddd \theta}{2\pi} .
\end{align*}
\qed
\section{Some useful entropies}
\label{appendix}
\begin{lemma}
\label{LLL}
For relative entropies with respect to $\UNIF$ we have
\begin{align}
\label{g=-1}
\mathcal K(\mu^{1,0}\mid\UNIF) &= 1- \log 2\\
\label{g=-1bis}
 \mathcal K(\mu^{1,1}\mid\UNIF) &= 1 - \log 2\\
 \label{g=0}
\mathcal K(\mu^{2,0}\mid \UNIF) &= \frac{7}{3} -
\log 6\, . 
\end{align}
Furthermore, we have for $|\g|\leq 1$
\begin{align}
\label{522}
\mathcal K(\UNIF \mid \mu_\g^{1,1}) &= - \log \left(\frac{1+ \sqrt{1-\g^2}}{2}\right)\,,\\
\label{522'}
\mathcal K(\mu_{1}^{1,1}\mid \mu_\g^{1,1}) &= 1 - \log 2 -  \log \frac{1+ \sqrt{1-\g^2}}{2} - \frac{\g}{1+\sqrt{1-\g^2}}\,,\\
\mathcal K(\UNIF\mid \mu_\g^{2,0}) 
 &= - \log\left(\frac{|\g|(1+\alpha)}{6(1-\alpha)}\right)\mathbbm{1}_{\{\g\neq 0\}} =:-K(\g) \,
 \end{align} 
 where 
\begin{align}
\label{defalpha1}\alpha= \sqrt{\frac{-a + \sqrt{a^2+4a}}{2}} \ , \  a= \frac{3}{2}(\g^{-1}-1)\,.
\end{align}
\end{lemma}
 
 \begin{proof}
 (\ref{g=-1}) is just $H(1)$ in (\ref{Hvalue}). So is (\ref{g=-1bis}) after the change of variable $\varphi = 2\theta$. 
 To prove (\ref{g=0}), notice that
\begin{align*}
\notag
\mathcal K(\mu^{2,0}\mid \UNIF) &= \int_0^{2\pi}\log \left(1- \frac{4}{3} \cos\theta + \frac{1}{3}\cos 2\theta\right) \frac{\ddd \theta}{2\pi}  -\tilde I(1)  \\
& = \log \frac{2 }{3} + 2 \int_0^{2\pi}\log (1- \cos\theta) \frac{\ddd \theta}{2\pi}  -\tilde I(1) ,
\end{align*}
with $\tilde I(1)$ as in \eqref{Ig1}. It remains to apply \eqref{logg}.

For (\ref{522}), apply the change of variable $\varphi = 2\theta$ and \eqref{logg}.
For (\ref{522'}), apply the same change of variable,  (\ref{logg}) and (\ref{logcosg}). 
To compute $K(\g)= - \mathcal K(\UNIF\mid \mu_\g^{2,0})$ notice that
\begin{align*}
\mathcal K(\UNIF\mid \mu_\g^{2,0}) = -\int_0^{2\pi} \log\left| \frac{\g}{6}h(e^{i\theta})\right| \frac{\ddd \theta}{2\pi}\,,
\end{align*}
with $h$ as in \eqref{defha}. Jensen's formula gives
\begin{align}
\notag
K(\g)&= \log |\g/6|-\log |\zeta_1| - \log |\bar \zeta_1|
=  \log |\g/6| - [\log (1-\alpha)^2 + \log(1 + \beta^2)]\\
\label{Jensen}
&= \log|\g/6| - [2\log (1-\alpha) + \log (1-\alpha^2)]= \log\frac{|\g|(1+\alpha)}{6(1-\alpha)}\,.
\end{align}
\qed
\end{proof}
\section{Auxiliary computations}
\label{sec:computations}

\subsection{Some auxiliary integrals}

The following integral is an easy consequence of Jensen's formula (see \cite[p. 138]{simonopuc1}): 
for $|\g| \leq 1$,
\begin{align}
\label{logg}
\int_0^{2\pi} \log (1 - \g \cos\theta) \frac{\ddd \theta}{2\pi} &= \log \frac{1+\sqrt{1-\g^2}}{2}\,.
\end{align}
Also the following integral is well-known (see e.g. \cite[Prop. 3.1]{breuersizei2018}): for $n\neq 0$,  
\begin{align}
\label{eipi}
\int_0^{2\pi} e^{in \theta}\log (1- \cos \theta) \frac{\dd \theta}{2\pi} = - |n|^{-1}  .
\end{align}
Let us give a third integral used many times in the computation of entropies.

\begin{lemma}
For $|\g|\leq 1$, we have
\begin{align}
\label{logcosg}
\int_0^{2\pi} \cos \theta \log (1- \g \cos\theta) \frac{\ddd \theta}{2\pi}
= -\frac{\g}{1+\sqrt{1-\g^2}}\,.
\end{align}
\end{lemma}

\proof
For $\g=0$, the statement is {\fab obvious}. For $\g\neq 0$, we integrate by parts to obtain
\begin{align*}
\int_0^{2\pi} \cos \theta \log (1- \g \cos\theta) \frac{\ddd \theta}{2\pi} =
 -\int_0^{2\pi} \frac{\g \sin^2 \theta}{1 - \g \cos \theta} \frac{\ddd \theta}{2\pi}
=\oint_\T f(z) \frac{\ddd z}{2i\pi}
\end{align*}
with
\[f(z) :=  \frac{\g}{2}\frac{(z-z^{-1})^2}{z[2- \g(z+z^{-1})]} = -\frac{1}{2} \frac{(z^2-1)^2}{z^2 [z^2-\tfrac{2}{\g}z +1]}\,.\] 
The poles of $f$ are $z_\pm=  \g^{-1} \pm \sqrt{\g^{-2}-1}$ and $0$ (double).
The residue in $0$ of $f$
is $-\g^{-1}$. 
If $\g > 0$ then $z_- \in \mathbb D$ and 
since $z_+z_- = 1$, the residue in $z_- \in \mathbb D$  is
\[-\frac{1}{2} \frac{(z_-^2 -1)^2}{z_-^2 (z_- -z_+)}= -\frac{1}{2}(z_- - z_+)= \sqrt{\g^{-2}-1}\]
so that
\[\oint_\T f(z) \frac{\ddd z}{2i\pi} = -\g^{-1}  + \sqrt{\g^2-1} = \g^{-1} (\sqrt{1-\g^2}-1)
\,. \]
If $\g< 0$ then $z_+ \in \mathbb D$ and the residue in $z_+$ is $-\frac{1}{2}(z_+ - z_-)= -\sqrt{1-\g^{2}}$ so that
\[\oint_\T f(z) \frac{\ddd z}{2i\pi} = -\g^{-1} - \sqrt{1-\g^2} = \g^{-1} (\sqrt{1-\g^2}-1)\,. \]
\QED

The following useful corollary is {\fab an obvious} consequence of (\ref{logg}) and (\ref{logcosg}).

\begin{corollary} 
 For $|\g| \leq 1$, 
 \begin{align}
 H(\g)& :=\int_0^{2\pi} (1 -\g \cos \theta)   \log (1 -\g \cos \theta)\  \frac{\ddd \theta}{2\pi}
 \label{Hvalue}
 =  1- \sqrt{1- \g^2} + \log \frac{1 + \sqrt{1-\g^2}}{2} \,.
  \end{align}
\end{corollary}

 \begin{lemma}
\label{LI}
Let \begin{align*}
I(\g) &=\int_0^{2\pi} \log \frac{1 - \frac{4}{3} \cos\theta + \frac{1}{3}\cos 2\theta}{1- \frac{4\g}{3} \cos\theta + \frac{\g}{3}\cos 2\theta}\left(- \frac{4}{3} \cos\theta + \frac{1}{3}\cos 2\theta\right)\frac{\ddd \theta}{2\pi}
\end{align*}
then for $0<\g\leq 1$ we have 
\begin{align}
\label{Ig}
I(\g) = \frac{2\alpha(\alpha^2+3\alpha +3)}{3(\alpha +1)}\,.
\end{align}
\end{lemma}

\proof
We have
\begin{align}
\label{718}I(\g) =  \tilde I(\g) - \tilde I(1)\,,
\end{align}
with 
\begin{align*}
\tilde I(\g) = - \int_0^{2\pi}\log \left(1- \frac{4\g}{3} \cos\theta + \frac{\g}{3}\cos 2\theta\right) \left(- \frac{4}{3} \cos\theta + \frac{1}{3}\cos 2\theta\right)\frac{\ddd \theta}{2\pi}
\end{align*}
For $|\g| < 1$ we get by integrating by parts
\begin{align*}
\notag
\tilde I(\g) &= 
\frac{4\g}{9}\int_0^{2\pi}\frac{\sin^2 \theta (1-\cos\theta)(-4+ \cos \theta)}
{\left(1- \frac{4\g}{3} \cos\theta + \frac{\g}{3}\cos 2\theta\right)}\frac{\ddd \theta}{2\pi}\\
\notag
&=\frac{\g}{12}\oint_\T \frac{(z-z^{-1})^2 (2-z-z^{-1}) (8- z - z^{-1})}{\left(3- 2\g(z+z^{-1}) + \frac{\g}{2}(z^2 + z^{-2})\right)} \frac{\ddd z}{2i\pi z}\\
&= \frac{1}{6}\oint_\T \frac{(z^2-1)^2 (2z-z^2-1) (8z- z^2 - 1)}{z^{3}(\tfrac{6}{\g}z^2-4z^3-4z+z^4+1)
} \frac{\ddd z}{2i\pi} 
= \frac{1}{6}\oint_\T f(z) \frac{\ddd z}{2i\pi}\,,
\end{align*}
where 
\begin{align*}
f(z) = \frac{(z^2-1)^2 (2z-z^2-1) (8z- z^2 - 1)}{z^{3}h(z)}
\end{align*}
and $h$ is given by 
\begin{align}
\label{defha}
h(z) = 
(z-1)^4 + 4az^2\ , \ a= \frac{3}{2}(\g^{-1}-1)\,.
\end{align} 
The roots of $h$ are solutions of 
\[(z-1)^2= \pm 2 iz\sqrt a\,,\]
or
\[z^2 -2(1\pm i \sqrt a)z +1 = 0\,.\]
If we set
\begin{align}
\label{defalpha}\alpha= \sqrt{\frac{-a + \sqrt{a^2+4a}}{2}} , \quad 
\beta = \sqrt{\frac{a + \sqrt{a^2+4a}}{2}} = \frac{\sqrt a}{\alpha}
\end{align}
these roots are
\begin{align}
\notag
z_1 = (1+\alpha) (1+i\beta) ,\quad  \zeta_1 =(1-\alpha) (1-i\beta) = \frac{1}{z_1}\,,\\
\label{theroots}
\bar z_1 = (1+\alpha) (1-i\beta)  , \quad \bar \zeta_1 = (1-\alpha)(1+i\beta) = \frac{1}{\bar z_1}\,.
\end{align}
Notice that  $|z_1| > 1$ and $|\zeta_1| < 1$.
The function $f$ has a triple pole in $0$ and the other poles in $\mathbb D$ are the roots $\zeta_1$ and $\bar\zeta_1$ of $h$. For $z$ small, 
\[z^3f(z) = \frac{1- 10 z + 16 z^2 +o(z^2)} {(1-4z+(6+4a)z^2 + o(z^2))}\,.\]
The residue in the triple pole $0$ is then  $-4a-14$. 
Let us compute the residue in $\zeta_1$.
We have
\begin{align*}
(z-\zeta_1)f(z) = \frac{(z - z^{-1})^2 \left(2-(z+z^{-1})\right)\left(8 - (z+z^{-1})\right)}{z^{-1}(z-z_1)(z-\bar z_1)(z-\bar \zeta_1)}\,.
\end{align*}
The numerator, evaluated in $\zeta_1$ is
\begin{align*}
\notag
(\zeta_1 - z_1)^2 \left(2-(\zeta_1+z_1)\right)\left(8 - (\zeta_1+z_1)\right)
=-16i\alpha\beta (\alpha +i \beta)^2 (3- i\alpha\beta)
\end{align*}
and the denominator is
\[\zeta_1^{-1}(\zeta_1- z_1) (\zeta_1 - \bar z_1) (\zeta_1  - \bar \zeta_1) =-8i \alpha\beta(1-\alpha^2)(1+\beta^2)
(\alpha+i\beta) = -8i \alpha\beta 
(\alpha+i\beta)\,,\]
so that the residue is
\begin{align*}
\operatorname{Res}(\zeta_1) = 2(\alpha+i\beta)(3-i\alpha\beta) \,.
\end{align*}
By symmetry,
\begin{align*}
\operatorname{Res}(\zeta_1) + \operatorname{Res}(\bar\zeta_1) = 2\mathrm{Re} \left(2(\alpha+i\beta)(3-i\alpha\beta)\right)
= 4\alpha(3+ \beta^2) = \frac{4\alpha(3-2\alpha^2)}{1-\alpha^2}\,. 
\end{align*}
Taking the sum of the residues we get
\begin{align*}
\notag
\oint_\T f(z) \frac{\ddd z}{2i\pi } &= -4a- 14 + \frac{4\alpha(3-2\alpha^2)}{1-\alpha^2}
=-4\frac{\alpha^4}{1-\alpha^2} -14 +  \frac{4\alpha(3-2\alpha^2)}{1-\alpha^2}\\ &= \frac{2(2\alpha^2 +8\alpha +7)(\alpha-1)}{\alpha+1} \,,
\end{align*}
hence
\begin{align*}
\tilde I(\g) = \frac{(2\alpha^2 +8\alpha +7)(\alpha-1)}{3(\alpha+1)}\,.
\end{align*}
For $\g=1$ we use \eqref{eipi} to obtain
\begin{align} \label{Ig1}
\notag
\tilde I(1) &= -\int_0^{2\pi} \log \left(\tfrac{2}{3} (1-\cos \theta)^2\right)\left(-\frac{4}{3}\cos \theta + \frac{1}{3}\cos 2\theta\right)\frac{\ddd \theta}{2\pi}\\
 &= -2 \int_0^{2\pi} \log (1-\cos \theta)\left(-\frac{4}{3}\cos \theta + \frac{1}{3}\cos 2\theta\right)\frac{\ddd \theta}{2\pi}
=-\frac{7}{3}\,,
\end{align}
which leads to (\ref{Ig}) using (\ref{718}).
\QED

\subsection{Computations for the Gross-Witten sum rule in the gapped case}
\label{app:remark}

In this section, we evaluate the sum rule with reference measure $\GW_\g$ with V-coefficients given in \eqref{alphalimGW+}. As explained in Remark \ref{rem:counterexample}, our aim is to show that the sum rule in Theorem \ref{Unitabstractnewsumrule}, written as 
\begin{align*}
\mathrm{LHS}(\mu) = \mathrm{RHS}(\mu), 
\end{align*}
does not hold for all (symmetric) measures $\mu\in \mathcal S_1(I)$.

For $\g>1$, $\GW_\g$ is given by
\begin{equation}\label{GWMeq}
\GW_\g (\ddd z) = \frac{ \g }{\pi}\sin(\theta/2)\!\  \sqrt{\sin^2(\theta_\g/2)-\cos^2(\theta/2)}\!\ 1_{[\pi-\theta_\g, \pi+\theta_\g]}\!\ \ddd\theta\,, 
\end{equation}
where $\theta_{g}  \in (0, \pi)$ is the solution of
\begin{equation*}\label{eqthetag1}
\sin^2 (\theta_\g /2) =  \g^{-1}\,.
\end{equation*}
For $\gamma\in [0,1)$, we consider the two measures $\GE_{- \gamma}$ and $\GE_{\gamma}$ with constant V-coefficients 
\begin{equation}
\alpha_k (\GE_{-\gamma})= -\gamma, \qquad \alpha_k (\GE_{\gamma})= \gamma 
\end{equation}
for $k\geq 0$ and as given in \cite{simonopuc1}, page 87.
The measure $\GE_{-\gamma}$ satisfies     
\begin{align} \label{eq:geronimusmeasure}
\GE_{-\gamma}(\ddd z) =\frac{1}{1-\gamma}\frac{\sqrt{\cos^2(\theta_\gamma /2) - \cos^2(\theta/2)}}{2\pi \sin(\theta/2)}1_{[\theta_\gamma, 2\pi-\theta_\gamma]} \,\ddd \theta,
\end{align}
where $\theta_\gamma =2 \arcsin  \gamma$. Notice that we have parametrized the Hua–Pickrell distribution in a convenient way, and that our parametrization corresponds to the parameter $\gamma/(1-\gamma)$ in the more standard formulation. 
The measure $\GE_{\gamma}$ with reflected coefficients  
is given by 
\begin{align*}
\GE_\gamma(\ddd z) =\frac{1}{1+\gamma}\frac{\sqrt{\cos^2(\theta_\gamma /2) - \cos^2(\theta/2)}}{2\pi \sin(\theta/2)}1_{[\theta_\gamma, 2\pi-\theta_\gamma]} \ddd \theta + \frac{2\gamma}{1+\gamma}\delta_0\,.
\end{align*}
To fit the essential support of $\GE_{-\gamma}$ and $\GE_{-\gamma}$ with the one of $\GW_\g$, we need
\[\pi-\theta_\g = \theta_\gamma\,,\]
which is equivalent to
\begin{align*}
\gamma = \sqrt{1-|\g|^{-1}}=-\mathfrak a \,.
\end{align*}
For this parameter, $\GE_{\pm\gamma} \in \mathcal S_1(I)$ and $\GE_{-\gamma}\in M_\V$. We have then
\begin{align*}
\frac{\ddd \GW_\g}{\ddd \GE_{-\gamma}}(z)&= \frac{1}{2(1+\gamma)}\sin^2(\theta/2) , \\
\frac{\ddd \GW_\g}{\ddd \GE_\gamma}(z)&= \frac{1}{2(1-\gamma)}\sin^2(\theta/2) .
\end{align*}
For $\GE_{-\gamma}$, the left hand side of the sum rule is 
\begin{align*}
\mathrm{LHS}(\GE_{-\gamma}) & = \mathcal K\left(\GW_\g\mid\GE_{-\gamma}\right)
 =  - \log (2 (1+\gamma)) + \int \log \sin^2(\theta/2) \GW_\g(\ddd z) .
\end{align*}
and for $\GE_{\gamma}$, we have
\begin{align*}
\mathrm{LHS}(\GE_{\gamma}) & = \mathcal K\left(\GW_\g\mid\GE_{\gamma}\right) + \mathcal F(0) \\
& =  - \log (2 (1-\gamma)) + \int \log \sin^2(\theta/2) \GW_\g(\ddd z)+ \mathcal F(0) . 
\end{align*}
The contribution of the outlier for $\GE_\gamma$ is
\[\mathcal F(0) =\frac{2}{1-\gamma^2}\int_0^{\theta_\gamma} \sin(\theta/2)  \,   \sqrt{\cos^2(\theta/2)-\sin^2(\theta_\g/2)}\,\ddd \theta\,.\]
Changing $\cos(\theta/2)= \sqrt{1-\gamma}\cosh \varphi$, we obtain
\begin{align*}
\mathcal F(0)  = \frac{4}{1-\gamma^2}\int_0^{\varphi_\gamma}\sinh^2 (\varphi)\,  \ddd\varphi\ ,
\end{align*}
where $\cosh (\varphi_\gamma) = 1/\sqrt{1-\gamma}$, and then
\begin{align}
\label{F(0)}
\mathcal F(0)  =\frac{1}{2}\left(\sinh\varphi \cosh\varphi - \varphi\right)\big|_0^{\varphi_\g} = \frac{1}{2}\left(\frac{\sqrt\gamma}{1-\gamma}- \log \frac{1+\sqrt \gamma}{\sqrt{1-\gamma}}\right)
\end{align}
This implies that 
\begin{align*}
\mathrm{LHS}(\GE_{\gamma}) - \mathrm{LHS}(\GE_{-\gamma}) = \log \left(\frac{1+\gamma}{1-\gamma}\right) + \frac{1}{2}\left(\frac{\sqrt\gamma}{1-\gamma}- \log \frac{1+\sqrt \gamma}{\sqrt{1-\gamma}}\right). 
\end{align*}
As the potential  is  $\V(z) = \g \mathrm{Re}(z) = \frac{\g}{2}(z+z^{-1})$,
we obtain, 
\begin{align*}
\tr \, \mathcal G_L(\alpha) = \alpha_0 - \sum_{k=1}^{L-1} \alpha_k \bar\alpha_{k-1} . 
\end{align*}
Comparing with the decomposition in Proposition \ref{propclue}, we see that 
\begin{align*}
F_-(\alpha_{[0]}) = \g \mathrm{Re} (\alpha_0),\qquad G(\alpha_{[k-1,k]}) = -\g \mathrm{Re} ( \alpha_k \bar\alpha_{k-1} ),\qquad F_-(\alpha_{[L-1,L-1]}) = 0 .
\end{align*}
For a symmetric measure $\mu$ with real coefficients $\alpha$, this implies
\begin{align*}
\mathrm{RHS}(\mu) = \g(\alpha_0 - \alpha_0^\g) - \sum_{k=1}^\infty\left(\g (\alpha_k\alpha_{k-1} - \alpha_k^\g\alpha_{k-1}^\g)  \
+ \log \frac{1- \alpha_k^2}{1-|\alpha_k^\g|^2}\right) .
\end{align*}
We have successively
\begin{align*}
\notag
\mathrm{RHS}(\GE_{\gamma})&=  \g(\gamma- \alpha_0^\g) - \sum_{k=1}^\infty\left(\g (\gamma^2 + \alpha_k^\g\alpha_{k-1}^\g)  \
- \log \frac{1- \gamma^2}{1-|\alpha_k^\g|^2}\right) , \\
\mathrm{RHS}(\GE_{-\gamma})&=  -\g(\gamma+ \alpha_0^\g) - \sum_{k=1}^\infty\left(\g (\gamma^2 - \alpha_k^\g\alpha_{k-1}^\g)  \
+ \log \frac{1- \gamma^2}{1-|\alpha_k^\g|^2}\right) , 
\end{align*}
such that 
\begin{align*}
\mathrm{RHS}(\GE_{\gamma}) - \mathrm{RHS}(\GE_{-\gamma}) = 2\g \gamma .
\end{align*}
By Corollary \ref{Unitnewsumrulesymmetric}, we know that the sum rule holds for $\GE_{-\gamma}$, so that 
\begin{align*}
\mathrm{LHS}(\GE_{\gamma}) \neq \mathrm{RHS}(\GE_{\gamma}) . 
\end{align*} }